\documentclass[a4paper,12pt,reqno]{amsart}
\usepackage{latexsym, mathabx}
\usepackage{amsmath,amsthm,amssymb,amscd}
\usepackage{fullpage}
\usepackage{xcolor}
\usepackage{parskip}
\usepackage{mathtools}
\usepackage{thmtools}
\mathtoolsset{showonlyrefs}
\usepackage{csquotes}
\usepackage{comment}
\usepackage{pdfpages}
\usepackage[activate={true,nocompatibility},final,tracking=true,kerning=true,spacing=true]{microtype}
\usepackage[pdfencoding=unicode,pdftex,bookmarks=true,bookmarksnumbered]{hyperref}
\definecolor{citecolour}{rgb}{0.0, 0.0, 0.8}
\colorlet{linkcolour}{green!50!black}
\hypersetup{colorlinks,breaklinks,
		linkcolor=linkcolour,citecolor=citecolour,
		filecolor=linkcolour, urlcolor=linkcolour}
\definecolor{paperblue}{RGB}{20,60,130}
\definecolor{paperpurple}{RGB}{95,55,135}
\usepackage{txfonts,pxfonts,tikz} 
\usepackage{tgschola}
\usepackage[T1]{fontenc}

\newtheorem{prevtheorem}{Theorem}

\newtheorem{theorem}{Theorem}

\newtheorem{proposition}{Proposition}

\newtheorem{lemma}[theorem]{Lemma}
\newtheorem{corollary}{Corollary}

\newenvironment{proofofA}{{\it {Proof of Theorem \ref{Thm:main}.} }}{\hfill $\blacksquare$ \\}
\newenvironment{proofofB}{{\it {Proof of Theorem \ref{thm:Lattices}.} }}{\hfill $\blacksquare$ \\}

\newenvironment{proofof}{{\bf {Proof.} }}{\hfill $\blacksquare$ \\}

\def\Z{\mathbf{Z}}

\def\ZZ{\mathbb{Z}}
\def\RR{\mathbb{R}}
\def\T{\mathbf{T}}
\def\vol{\text{vol}}

\usepackage{graphicx} 

\title{Three subgroup common transversal in abelian groups} 
\author{Maria Loukaki }
\address{Department of Mathematics and Applied Mathematics, University of Crete, Voutes Campus, 70013 Heraklion, Greece}
\email{mloukaki@uoc.gr}

\author{Emmanouil Spyridakis  }
\address{Department of Mathematics and Applied Mathematics, University of Crete, Voutes Campus, 70013 Heraklion, Greece}
\email{manos.ch.spyridakis@gmail.com}

\date{June 2026}

\begin{document}

\begin{abstract}
We completely characterize when three equal-order subgroups of a finite abelian group share a common transversal. Consequently, we determine when three full-rank lattices of equal volume  with a discrete sum  in $\mathbb{R}^d$ admit a common fundamental domain, answering, for $n=3$, an open  question from 1997.
\end{abstract}



\maketitle

\section{Introduction}

In the 1950s, Steinhaus~\cite{Sie58} asked whether there exists a planar set that {\it tiles} the plane when {\it translated } by any full-rank lattice obtained by rotating the integer lattice around the origin. 
In the language used in this paper, Steinhaus asked if there exists a common transversal $E$ of all full-rank lattices $L_{\theta}= R_\theta\mathbb{Z}^2$, 
where $\theta \in [0, 2\pi)$ and $R_\theta$ denotes rotation by $\theta$ around the origin. That is,  whether there exists a set $E \subseteq \RR^2$ that covers all  cosets modulo each  $L_{\theta}$ exactly once. 
An affirmative answer to the Steinhaus problem  was given only in  this century~\cite{Jackson}. But the variation of the problem  where the subset of the plane  is asked to be Lebesgue measurable 
has sparked much more interest and is still open. The best result up to date  for the measurable problem can be found in~\cite{KW98}. 

The Steinhaus problem extends naturally to $\mathbb{R}^d$ for $d > 2$. A full-rank lattice $L$ in $\mathbb{R}^d$ is a discrete subgroup of $\RR^d$. Every full-rank lattice $L$ in $\mathbb{R}^d$ can be expressed as $A\mathbb{Z}^d$ for some invertible matrix $A$, with volume $\text{vol}(L) = |\det(A)|$. A {\it transversal} (or a {\it fundamental domain} ) of $L=A\mathbb{Z}^d$ in $\mathbb{R}^d$ is a set containing  exactly one representative from each coset of $L$ in $\mathbb{R}^d$.
A common example of a fundamental domain of $L$ is the so-called fundamental parallelepiped of $L=A\ZZ^d$
\begin{equation*}
    P_L = A[0,1)^d
\end{equation*}
 Recent results in~\cite{grepstad2025bounded} and~\cite{spyridakis2026} demonstrate the existence of a bounded and measurable common fundamental domain for two full-rank lattices in $\mathbb{R}^d$ of equal volume.

This raises the question: Does a bounded and measurable common fundamental domain exist for three full-rank lattices of equal volume? Let us remark here that the question of finding a  set $E \subseteq \RR^d$  that is a common transversal for a finite number of full-rank lattices in $\RR^d$ first appeared in~\cite{kol97}. In the same paper Kolountzakis   gave  an example of three full-rank lattices of equal volume that do not admit a common fundamental domain. That being said, these full-rank lattices are not rotations of the full-rank lattice of the integers.  
Furthermore,  in~\cite{kol97} the following problem   was  first posed:  
\begin{center}
When does a  finite set of subgroups  $A_1, \ldots, A_n$ of the same finite index in an abelian group $G$  admit a common transversal? 
\end{center}
It is worth mentioning that when $n=2$, Hall's Marriage Theorem guarantees the existence of a common transversal for two subgroups of the same index in $G$. But when more than two subgroups are involved common transversals  do not always exist. Indeed take $G= \Z_p \times \Z_p$, for any prime $p$, and consider all its subgroups of order $p$. These are $p+1$ in total and their union is equal to $G$. They cannot share a common transversal because no element of $G$ lies outside their union. Thus, any common transversal (which, after normalization, can be assumed to contain the $0$ element of $G$) would contain at least one more element from one of the subgroups.

Nevertheless, it was proved \cite[Thm.~1]{kol97} that if the sum  $A_1+ \ldots +A_n$  of the subgroups, is direct then they always share a common transversal in $G$. (We say that a sum of subgroups  $A_1+ \ldots +A_n$ is direct when every element $a \in A_1 +\cdots  +A_n$ can be written in a unique way as $a= a_i+\cdots +a_n $,  with $a_i \in A_i$).   Consequently it was shown~\cite[Thm.~4]{kol97} that if the full-rank lattices $L_1, \cdots, L_n$ of $\RR^d$ have the same volume and both $\sum L_i$ and $\sum L_i^*$ (the sum of their duals) are direct sums then they have a Lebesgue measurable common transversal, while if only the sum of $L_i$ is direct then a bounded (but not necessarily measurable) common transversal was constructed~\cite[3.2]{kol97}.
The  question that was also  posed in~\cite[3.1]{kol97}, and is  of interest in this paper,  is whether full-rank lattices $L_1, \cdots, L_n$ sharing a common transversal can be characterized when they do not sum directly. 

Partial results for the  first question  were given in~\cite{ALS24} under the hypothesis that  the   subgroups $X, Y, Z$ are all cyclic (see~\cite[Thm. B]{ALS24}) or  that $X, Y, Z$  are all complemented in $G$ (see~\cite[Thm. A]{ALS24}).  
In particular, after factoring out by their intersection, it was proved in~\cite[Thm. B]{ALS24}  that three cyclic subgroups $X, Y$, and $Z$  lack a common transversal in their sum $G$ if and only if they exhibit a structure analogous to the Klein group. That is,  precisely when their Sylow $2$-subgroups $X_2, Y_2, Z_2$ are all cyclic, isomorphic to some $\mathbb{Z}_{2^n}$, and the Sylow $2$-subgroup of $G$ is $G_2 \cong  \mathbb{Z}_{2^n} \times \mathbb{Z}_{2^n}$, the  direct sum of any two of $X_2, Y_2$ and $ Z_2$.


The main result of this paper  fully resolves the problem of the existence of a common transversal for three subgroups of equal order inside a finite abelian group. In particular, we prove the the only problematic scenario where three subgroups $X, Y, Z$ lack a common transversal in $G$ is precisely the one that appeared  in the cyclic case, i.e. when (after factoring out by their intersection) the Sylow $2$-subgroups $X_2, Y_2, Z_2$ are all non-trivial cyclic  and $G_2$ is the direct sum of any two of them. Hence,  the existence of a common transversal for $X,Y,Z$ in $G$ depends solely on the nature and position of their Sylow $2$-sugroups within $G$.
Our main result reads: 
\begin{prevtheorem}\label{Thm:main}
Let $G$ be a finite abelian group with subgroups $X, Y, Z$ of equal index, and let $N = X \cap Y \cap Z$. Then $X, Y, Z$ have no common transversal in $G$ if and only if $G/N$ has even order and the Sylow $2$-subgroup $(G/N)_2 $ of $G/N$ satisfies  
\[
(G/N)_2 = (X/N)_2 \oplus (Y/N)_2 = (X/N)_2 \oplus (Z/N)_2 = (Y/N)_2 \oplus (Z/N)_2,
\]
where $(X/N)_2$, $(Y/N)_2$, and $(Z/N)_2$ are {\em cyclic }  Sylow $2$-subgroups of $X/N$, $Y/N$, and $Z/N$, respectively.
\end{prevtheorem}

 If the index $ [G:X]$  is odd, then $ G_2 \leq X_2$ and similarly $G_2 \leq Y_2, Z_2$. Hence $G_2 \leq N$, so $G/N$  has odd order and the obstruction in \autoref{Thm:main} cannot occur. Thus
\begin{corollary}
  Subgroups $X, Y, Z$ of a finite abelian group $G$ with equal odd index always share a common transversal in $G$.  
\end{corollary}

The proof of the theorem above   reduces to a perfect matching problem,   utilizing abelian group maps that share characteristics similar to those of a {\it complete map}. A bijection  $\theta: G \to G$ is a  complete mapping  of $G$  if the mapping $x \to x\theta(x)$ is also a bijection. It was shown by  Hall and Paige in 1955 (\cite{hall1955complete})
that any finite group with nontrivial cyclic Sylow $2$-subgroup does not admit complete mappings  and by Wilcox-Evans-Bray  in 2009 (\cite{WILCOX20091407}, \cite{EVANS2009105}) that the converse also holds. 
The primary  challenge with the maps needed for the proof of our Theorem A,  lies with the prime $p=2$
 as detailed in Section  \ref{sec:Bijections}.

Let us now turn our attention to full-rank lattices and see how our result on finite groups translates.
Consider three full-rank lattices $K,L,M\subseteq \RR^d$ of equal volume such that $ L+M+K= \ZZ^d $, via a non-singular linear transformation.  In other words,  $K,L,M$ are three full-rank lattices of equal volume with a discrete sum in $\RR^d$. 
Observe that in this case the sublattice $H:= K\cap L\cap M$ is also a full-rank lattice in $\RR^d$. Indeed,   the map
\begin{equation*}
    F: \ZZ^d/H \to  (\ZZ^d/K) \times (\ZZ^d/ L)\times (\ZZ^d/M)
\end{equation*}
defined as  $F(n + H) = (n +K, n+L, n+ M)$ for every $n\in \ZZ^d$ is injective. Thus
\begin{equation*}
    [\ZZ^d:H]=|\ZZ^d/H|\le |\ZZ^d/K| |\ZZ^d/L| |\ZZ^d/M|= \vol(K) \vol(L) \vol(M)<\infty
\end{equation*}
since $K,L,M$ are full-rank lattices.

Using  \autoref{Thm:main} now we get 
\begin{prevtheorem}\label{thm:Lattices}
    Let $K,L,M$ be three full-rank lattices on $\RR^d$ of equal volume such that $K+L+M$ is a full-rank lattice. Denote their intersection by $H= K\cap L\cap M$ and the quotient groups $X:=K/H$, $Y:= L/H$, $Z:= M/H$ inside $G:= (L+M+K)/H$. Then $K,L,M$ do not share a common fundamental domain in $\RR^d$ if and only if the  Sylow $2$-subgroups of $X,Y,Z$  are non-trivial cyclic, pairwise disjoint satisfying 
    $G_2 = X_2 \oplus Y_2= X_2 \oplus Z_2 = Y_2 \oplus Z_2$. 
\end{prevtheorem}
Finally, let us remark that if $K, L$, and $M$ share a common fundamental domain within the full-rank lattice $K+L+M$, then a common fundamental domain for $K, L$, and $M$ in $\mathbb{R}^d$ can be constructed as a finite union of polytopes. Specifically, if $F$ is a finite common transversal of $K, L$, and $M$ in $K+L+M$, then $F+P$, where $P$ is a fundamental parallelepiped of the full-rank lattice $K+L+M$, serves as such a common fundamental domain in $\mathbb{R}^d$ and is a finite union of polytopes.\\

Throughout this paper, $G$ denotes a finite abelian group, and $X, Y, Z$ are subgroups of $G$ of the same order,  whose common transversals are the subject of investigation. We denote by  $\T_G(X, Y, Z)$ the set of common transversals of $X, Y, Z$ in $G$. We also write $G_p$ for the (unique as $G$ is abelian) Sylow $p$-subgroup of $G$ for any prime $p $ that divides the order $|G|$.

This paper is structured as follows:  Section \ref{sec:Bijections} establishes necessary bijections for subsequent proofs. We make some reductions and provide the notation we will be using in Section \ref{sec:Reductions}. Sections \ref{sec:DirectSum}, \ref{sec:sumOfTwo}, and \ref{sec:GeneralCase} focus on proving first some special cases and ultimately the general  case  of our  main result (Theorem A) for abelian $p$-groups. This approach, on $p$-groups instead of general abelian groups,  leverages the bijections from Section  \ref{sec:Bijections} and allows order inequalities to translate into divisibility relations, ensuring the existence of subgroups of the desired order. Specifically, Theorem A is proven for the special case of an abelian $p$-group $G = X \oplus Y$ in Section \ref{sec:DirectSum}, and for $G = X+Y$ in  Section  \ref{sec:sumOfTwo}. The proof of Theorem A for a general $p$-group $G = X+Y+Z$ is presented in Section \ref{sec:GeneralCase}.  In Section \ref{sec:ThmA} we complete the proofs of Theorem A and B by demonstrating that three subgroups $X, Y, Z$ in a finite abelian group $G$ share a common transversal if and only if their Sylow $p$-subgroups $X_p, Y_p, Z_p$ share a common transversal in $G_p$ for every prime divisor $p$ of $|G|$. Finally, Section  \ref{sec:Lattices} provides examples of full-rank lattices that do not share a common transversal.

\section{Bijections}\label{sec:Bijections}
While common transversals exhibit algebraic structure in special cases, such as when the underlying subgroups are complemented, their existence generally reduces to a perfect matching problem, and as such have been treated in this paper. 
This section serves exactly this purpose; to
provide all the required machinery, through bijections, that will be needed in the
forthcoming proofs. 
The starting point, being Hall's Theorem~\cite{Hall1952ACP} according to which if $A=\{ a_i \mid i \leq n\} $  is an abelian finite group of  order $n$ and $\{ b_1, \cdots ,b_n\} $  are elements of $G$ (not necessarily distinct )  with $b_1+ \cdots +b_n=0$
 then there exists a permutation $\sigma \in S_n$ with 
 \[ \{ a_{\sigma(i)}+b_i \mid  i\leq n \}=A = \{a_1, \cdots , a_n\}.
 \]
 Additionally, we recall Hall-Paige Theorem (we only need it here for abelian groups)  according to which every finite group $G$  admits a complete map (i.e.,  a bijection $\theta: G \to G$ with $x+\theta(x)$ being also a bijection) if and only if $G$ has a trivial or non-cyclic Sylow $2$-group.
For a detailed analysis and the history of the Hall-Paige Theorem  one might see~\cite{Evans2018}. 
The maps presented in this section, are based on Hall's or Hall-Paige theorem and share a similar structure. But they  are weaker forms of complete maps, with cyclic $2$-groups often presenting a challenge.

Let us note  that for odd order groups, all the required maps constructed below  could simply replaced by the bijection $\phi: G \to G$ where $\phi(x) = -x$. However, the prime 2 necessitates more complex constructions. Therefore, while keeping in mind the simplified bijection available for odd order groups, we state and prove the lemmas and propositions of this section in a general way, regardless of parity.

Our first result of this type, that is actually a straightforward application of Hall's theorem, is the following.
\begin{lemma}\label{lem:2-bijections}
Assume $G$ is an abelian $p$-group  and let $1 \neq  H  < G$ be a proper  subgroup of $G$ and  $S $ a transversal of $H$ in $G$, so that every $g \in G$ is written uniquely as $g = s + h$ for some  $s \in S $ and $h  \in H$. Then there exists a bijection $\psi: G \to G $  so that $\psi(s + h)-h$ is also a bijection of $G$. 
If $G $ is not a cyclic $2$-group,  then there exists a bijection  $\phi:G  \to G $ so that $\phi(s +h )-s $ is also a bijection of $G$.    
\end{lemma}

\begin{proof}
Let $S= \{s_i \mid i \leq n\} $ and   $H = \{ h_j \mid j \leq |H| \}$. For every element in $g_{i,j} \in G $,   we write  $g_{i,j} = s_i+ h_j $ with $s_i \in S$ and $ h_j \in H$ uniquely determined. Let $B = \{ b_{i, j} = -h_j \mid i\leq n, j\leq |H| \}$ be the sequence where each $-h_j \in H$ appears $[G:H]$ times. Then
\[
\sum_{\substack{i \leq n, \\ j \leq |H|}} {b_{i,j}} = [G:H] \sum_{j \leq |H|} (-h_j) = [G:H] \sum_{j \leq |H|} h_j= 0.
\]
The last equality holds because $\sum_{j  \leq |H|} h_j = 0$, unless H is a cyclic $2$-group, in which case the sum is its unique involution. But then  $[G:H] \sum_{j \leq |H|} h_j = 0$ since $[G:H] \geq 2 $ is an even number.
Now we apply Hall's Theorem (Main Theorem in~\cite{Hall1952ACP}) to the set $B$. Hence there exists a bijection $\psi :G \to G $ so that 
\[
\psi(g_{i,j})+b_{i,j}= \psi(g_{i,j}) - h_j, 
\]
is a bijection of $G$ (where $g_{i,j}= s_i +h_j$). 

Assume now that $G$ is not a cyclic $2$-group, and  consider the sequence  $B'= \{ b'_{i, j} = s_i \mid i\leq n, j\leq |H| \}$. Observe that $\sum_{\substack{i \leq n, \\ j \leq |H|}} g_{i, j} = 0$, because $G$ is not a cyclic $2$-group. Hence 
\[
0=\sum_{i, j} g_{i, j}= \sum_{i,j} s_i+h_j= |H|\sum_{i \leq n } s_i +[G:H] \sum_{j \leq |H|} h_j= |H|\sum_{i \leq n } s_i.
\]
The last equality holds because $\sum_{j \leq |H| } h_j= 0$ unless $H$ is a cyclic $2$-group, in which case the sum is its unique involution. In the latter case, $[G:H] \sum_{j \leq |H|} h_j = 0$ since $[G:H] \geq 2  $  is even. 
Therefore,  
\[
\sum_{i \leq n, j \leq |H|} b'_{i,j}= |H| \sum_{i \leq n} s_i = 0.
\]
Now we apply the main Theorem in~\cite{Hall1952ACP} to the set $B'$. Hence there exists 
a bijection $f':G \to G$ so that 
\[
f'(g_{i,j}) +b'_{i,j}= f'(g_{i,j}) +s_i
\]
is a bijection of $G$ (where $g_{i,j}= s_i+h_j$).

Then the function  $\phi : G \to G$ defined as 
\[
\phi(g_{i,j})= f'(g_{i,j})+s_i,
\]
 is clearly the desired bijection. 
This completes the proof of the lemma.
\end{proof}

Let us note that the requirement of the second part of the lemma that wants   $G$ to be  a non-cyclic $2$-group is essential; if $G$ is a cyclic $2$-group, there is no bijection $\phi:G\to G$ with $    \phi(s+h) - s$ being also a bjection of $G$. 
 Indeed, a necessary condition for such a bijection to exist is that
\begin{equation*}
    \sum_{g\in G}g=\sum_{s\in S, h\in H} (\phi(h+s)-s)= \sum_{s\in S, h\in H} \phi(h+s) - \sum_{s\in S, h\in H} s= \sum_{g\in G} g + |H| \left(\sum_{s\in S} s\right)
\end{equation*}
where the last equality holds because $\phi$ is a bijection on $G$. Thus, we must have that
\begin{equation*}
    |H| \left(\sum_{s\in S} s\right)=0
\end{equation*}
Since $G$ is a cyclic $2$-group and $H$ a proper subgroup, $G/H$ is a non-trivial cyclic $2$-group. Consequently, $\sum_{s\in S}s$, where $S$ is a complete system of representatives, represents the unique coset of order $2$ in $G/H$. Every element in this coset has order $2|H|$, thus $|H| \left(\sum_{s\in S} s\right)\neq 0$. This contradiction establishes that the  requirement of a non-cyclic 2-group $G$ is necessary for the second part of the lemma.

Our next result reads:
\begin{lemma}\label{lem:bijectionInQuotient}
Assume $G$ is an abelian group $p$-group   and $1< A \leq  B \leq   G$  subgroups of $G$. Let $\{ g_i  \mid i \leq n  \} $ a transversal of $B $ in $G$.  Assume further that $ \{ b_j \mid j \leq r\} $  is a transversal of $A $ in $B$ and  $A = \{ a_t \mid t \leq s\}$.   Then there exists a bijection $f : G \to G$  so that for any fixed $t \leq s$ we have 
\[
f(g_i +b_j +a_t)- b_j  \not \equiv f (g_{i'} +b_{j'}+a_t) -b_{j'}  \pmod A
\]
for all $(i, j) \neq (i', j')$.  
\end{lemma}

\begin{proof}
Assume first that $A=B$ and so $r=1$ and $\{b_j\}= \{0\}$. 
Then the identity map $f:G \to G$ with $f(x)=x $ satisfies the lemma. Indeed for any fixed $t \leq s$ we get $f(g_i+a_t) = g_i +a_t \equiv g_i \pmod A$. Since $\{g_i \mid i \leq n\}$  is a transversal of $A=B$ in $G$, the cosets  $g_i +A$ are distinct.  Hence the required inequivalence holds for all $i \neq i'$. 

Hence we may   assume henceforth  that  $A \neq B$.
If $B/A $ is not a cyclic $2$-group, denote its elements by $\{ \bar{b}_j \mid j \leq r\}$.   Then $B/A$ satisfies the Hall Paige theorem and thus  there exists  a permutation $\sigma \in S_r$ that simultaneously satisfies  $\{ \bar{b}_{\sigma(j)} \mid j \leq r  \} = B/A$ and  $\{ \bar{b}_{\sigma(j)} - \bar{b}_j \mid j \leq r \}= B/A$.
  If $A= \{a_t \mid t \leq s\} $ then every element of $G$ is written uniquely as $g_i+b_j +a_t$ for some $i, j, t$. We  define $f:G \to G $ as 
  \[
f(g_i+b_j +a_t) = g_i +b_{\sigma(j)} +a_t, 
\]
for all $i \leq n, j \leq r $ and $t \leq s$. 
Clearly $f$ is a bijection of $G$, as  $\sigma $ is a permutation.  Furthermore, 
if $f(g_i +b_j +a_t )- b_j  \equiv f(g_{i'} +b_{j'} +a_t) -b_{j'}  \pmod A$ then 
$g_i +b_{\sigma(j)} -b_{j}  \equiv g_{i'} +b_{\sigma(j')} - b_{j'}  \pmod A$ forcing $i=i'$  (because  $g_i -g_{i'}  \in B$).  Hence we get $b_{\sigma(j)} -b_{j}  \equiv b_{\sigma(j')} - b_{j'} \pmod A$ and thus $\bar{b}_{\sigma(j)} - \bar{b}_j = \bar{b}_{\sigma(j')} - \bar{b}_{j'}$ contradicting the choice of $\sigma$.  So  the lemma follows in this case.

We now consider the remaining  case where $B/A$ is a non-trivial  cyclic $2$-group. Assume $B/A$ is cyclic of order $r=2^m$, generated by $\bar{b}_0$ for some $b_0 \in B$. So 
$B/A = \{  j \bar{b}_0 \mid j \leq r\}$ and $j b_0$ are all distinct modulo $A$, for $j \leq 2^m$, while $d:= 2^m b_0\in A$ (so $d$ is one of the $a_t$). Next, we partition $A$ into two disjoint subsets arbitrarily, $A_1$ and $A_2$, each of size $s/2$ (since $|A|$ is even). Let $\theta: A_1 \to A_2$ be any bijection. 

Consider the map $f : G \to G $ defined as 
\[
f(g_i+j b_0 +a_t)= \begin{cases}
   g_i+ 2j\, b_0 +\theta(a_t) +d,   &\text{for  }  1\leq j \leq 2^{m-1}, \, \,   a_t \in A_1\\
  g_i+ 2j\, b_0 +a_t,  &\text{for  }  2^{m-1} +1\leq j \leq 2^{m}, \, \,   a_t \in A_1\\
   g_i+(2j-1)b_0 +\theta^{-1}(a_t) +d,   &\text{for  }  1\leq j \leq 2^{m-1}, \, \,   a_t \in A_2\\
  g_i+ (2j-1)\, b_0 +a_t,  &\text{for  }  2^{m-1} +1 \leq j \leq 2^{m}, \, \, a_t \in A_2  
\end{cases} 
\]
for every $i \leq n$. 
We show that $f$ is surjective. 
So let $g\in G$. Since the set $\{g_i + jb_0 \: | \: i\le n, j\le r\}$ is a transversal of $A$ in $G$, the element $g$ is uniquely written in the form of:
\begin{equation*}
    g:= g_i + jb_0 + a
\end{equation*}
for some $i\le n, j\le r$ and $a\in A$. As  $d\in A$, the set $\{ a_t + d\: | \: t\le s\} = A$ and so 
\begin{equation}\label{eq:a}
    a:= a_t + d
\end{equation}
for exactly one $t\le s$. 

Let's first assume that $j\in \{1,...,2^m\}$ is even. 
If $a_t \in A_1$ then the  image of $g_i + (2^{m-1}+\frac{j}{2})b_0 + a_t \in G$ via $f$ is  
\begin{equation*}
    f(g_i + (2^{m-1}+\frac{j}{2})b_0 + a_t) = g_i +(2^m + j) b_0 +a_t = g_i + jb_0 + a_t+d= g_i + jb_0 + a= g.
\end{equation*}
If $a_t \in A_2$, there exists $a_{t'} \in A_1$ with $t'\le s$ such that $\theta(a_{t'}) = a_t$. Then
\[
f(g_i + \frac{j}{2} b_0 + a_{t'}) = g_i + jb_0 + \theta(a_{t'}) +d = g_i + jb_0 + a_t +d = g.
\]
Similarly we work for $j$ is odd. If  $a_t \in A_1$   we get 
\begin{align}
  f(g_i + \frac{j+1}{2}b_0 + \theta(a_t)) &=  g_i +(2\frac{j+1}{2} -1) b_0 +\theta^{-1}(\theta(a_t)) +d \\ &= g_i + jb_0 + a_t+d= g_i + jb_0 + a= g. 
\end{align}
While if $j$ is odd, and $a_t \in A_2$ then 
\[
 f(g_i + (2^{m-1}+\frac{j+1}{2})b_0 + a_t) =  g_i +(2^m +j) b_0 +a_t = g. 
  \]
  Hence $f$ is surjective and thus a bijection.
\begin{figure}[h]
\includegraphics[width=0.9\textwidth]{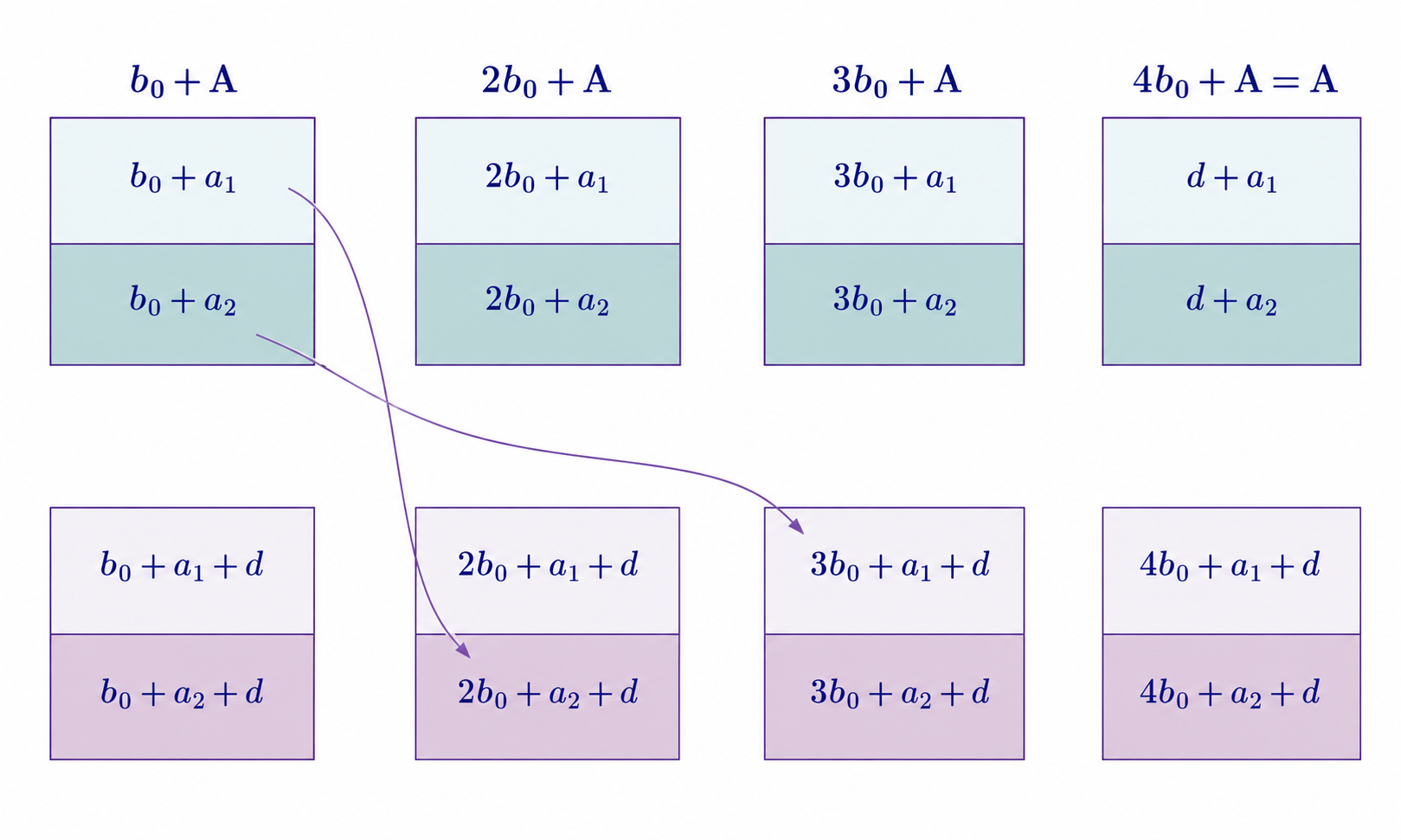}
\caption{The function $f$ matches the elements of each coset of $A=\{ a_1, a_2\}$ in the case where $G= B=\{ jb_0 + a_t \: | \: j\le 4, t\le 2\}$. The set $A_1:= \{a_1\}$,  $\theta(a_1)=a_2$ and $4b_0=d \in A$.}
\label{fig:myfigure}
\end{figure}

To prove the stated inequivalence modulo $A$,  we fix  $t\le s$ and  assume first  that $a_t\in A_1$. Then for every $i\le n$ and $ j\le r$ we have 
\begin{equation*}
    f(g_i +jb_0 +a_t)- jb_0 \equiv g_i +jb_0  \pmod A.
\end{equation*}
This implies that for each fixed $t$, with  $a_t \in A_1$, the set $\{f(g_i + jb_0 +a_t)- j b_0\: | \: i\le n, j\le r\}$ is a transversal of $A$ in $G$ and thus 
\begin{equation*}
    f(g_i +jb_0 + a_t) - jb_0 \not \equiv  f (g_{i'} +b_{j'}+a_t) -b_{j'}  \pmod A
\end{equation*}
for $(i, j) \neq (i', j')$, as desired.

Similarly if $a_t \in A_2$ then 
\begin{equation*}
    f(g_i + jb_0 +a_t) - jb_0 \equiv g_i + (j-1)b_0   \pmod A
\end{equation*}
for  all $i\le n$ and $ j\le r$. The set $\{g_i + (j-1)b_0 \: | \: i\le n , j\le r\}$ forms also  a transversal of $A$ in $G$, ensuring the inequivalence
\[
 f(g_i +jb_0 + a_t) - jb_0 \not \equiv  f (g_{i'} +j'b_0+a_t) -j'b_0  \pmod A
\]
for $(i, j) \neq (i',j')$. This completes the proof of the lemma.

\end{proof}

Let us remark that it is crucial in \autoref{lem:bijectionInQuotient} that $A$ is a non-trivial subgroup of $G$. For example, if $A=\{0\}$ and $B=G$ is a cyclic $2$-group the function sought by the lemma is a complete map. However, cyclic $2$-groups are known not to admit one.

   Furthermore, the assumption that $G$ is a $p$-group cannot be replaced by an  arbitrary abelian group; take for example the group $G=B=\Z_6$ and $ \Z_3 \cong A < B$. In this case $\{ b_j \mid j \leq 2\}=\{0, 1\} $ and the conclusion of the lemma states that for every $a \in A=\{0, 2, 4\} $ we should have $f(a) +A \neq f(a+1)-1 +A$ for some bijection $f: G \to G$. But  $G/A$ has order 2 and so the conclusion of the lemma equivalently translates to:
   $f(a) \in A$ if and only if $f(a+1) \in A$. 
   As $|A|=3$  one can see that this is impossible (because of parity)  for any  bijection $f$.

We complete the section with the following proposition.

\begin{proposition}\label{prop:f-transversal}
Let $G$ be an abelian $p$-group.    Assume $A\oplus B \leq G$ with $|A|=|B|=r>1$ and $B=\{ b_j \mid j \leq r\}$. Assume further that  $\{g_i \mid i \leq n \}$  a transversal of $A\oplus B$ in $G$. Then there exists a function $f : G \to G$  such that 
\begin{align*}
f(g_i+b_j ) &\not \equiv f(g_{i'}+b_{j'}) \pmod A \\
f(g_i+b_j )-g_i &\not \equiv f(g_{i'}+b_{j'})-g_{i'} \pmod B 
\end{align*}
 for all $(i, j ) \neq (i', j') $.
\end{proposition}

\begin{proof}
Let  $A=\{ a_j \mid j \leq r\}$ and $B = \{ b_j \mid j \leq r\}$  the elements of $A$ and $B$ in an arbitrary but fixed order. We write $\bar {g}_i $ for the elements of the quotient group $G/(A\oplus B) $ and we also fix their order  and their representatives $\{ g_i \mid i \leq n\}$ (that is,  the enumeration of $g_i$ is also  fixed).

If $G = A \oplus B$, then $n = 1$ (that is, $\{g_i\} = \{0\}$), and the required map is $f : G  \to G$ defined by $f(b_j + a_t) = f(b_j) = b_j + a_j$ for all $j, t \leq r$. Clearly, $f(b_j) \equiv b_j \pmod A$, so $f(b_j)$ are distinct for different values of $j$. Likewise, since $f(b_j) - 0 = f(b_j) \equiv a_j \pmod B$, they are also distinct modulo $B$ for different $j$. This proves the proposition in this case. Therefore, from now on we assume $G \neq A \oplus B$.

Define $H := G/(A \oplus B) \times A$. Then the elements of $H$ can be written as
\[ H = \{ h_{i, j}= (\bar{g}_i, a_j) \mid i \le n,\ j \le r \}. \]
The group $H$ is evidently non-cyclic, 
so the Hall–Paige theorem applies.  Hence there exists a bijection  $\phi: H \to H $ so that $\phi(h)+h$ is also a bijection of $H$.  Set 
\[
(\bar{g}_{\phi_{1,{i,j}}(i)} , a_{\phi_{2,{i,j}}(j)} ):= \phi( ( \bar{g}_i, a_j )), 
\]
for indices  $\phi_{1, i,j}(i) \leq n$ and $\phi_{2,i,j}(j) \leq r$. 
 That $\phi$ is a bijection forces  
 \begin{equation} \label{eq:phi-bijec}
 \{ (\phi_{1,{i,j}}(i), \phi_{2,{i,j}}(j) ) \mid i \leq n, j \leq r \} = \{ (i, j ) \mid i \leq n, j \leq r \}. 
\end{equation}
We now want to reformulate the fact that $\phi(h)+h$ is a bijection in terms of indices. To begin, for each $i \leq n$ and $j \leq r$ we set
\begin{align}\label{eq:theta-ij}
(   \bar{g}_{\theta_{{1, i,j}}(i)}, a_{\theta_{2,i,j}(j)}) &:= \phi( ( \bar{g}_i, a_j )) + (\bar{g}_i , a_j) \\
&=   (\bar{g}_{\phi_{{1, i,j}}(i)} , a_{\phi_{{2, i,j}}(j)} ) +   (\bar{g}_i , a_j)=   (\bar{g}_{\phi_{{1, i,j}}(i)} +\bar{g}_i ,  \,  a_{\phi_{{2, i,j}}(j)} +a_j ), 
\end{align}
for suitable indices  $\theta_{1, i,j}(i) \leq n$ and $\theta_{2, i,j}(j) \leq r$. That the map $\phi(h)+h$ is a bijection is equivalent to 
\begin{equation}\label{eq:theta-bijec}
\{ (\theta_{1, i,j}(i), \, \,  \theta_{2, i,j}(j)) \mid i \leq n,\; j \leq r \}
=
\{ (i, j) \mid i \leq n,\; j \leq r \}.
\end{equation}
Moreover, \eqref{eq:theta-ij} yields $\bar{g}_{\theta_{1,i,j}(i)} = \bar{g}_{\phi_{1,i,j}(i)} + \bar{g}_i$, and hence
\begin{equation}\label{eq:hat{a}}
g_i + g_{\phi_{1,i,j}(i)} = g_{\theta_{1, i,j}(i)} + \widehat{a}_{i,j} + \widehat{b}_{i,j}
\end{equation}
for some  $\widehat{a}_{i,j} \in A$ and $\widehat{b}_{i,j} \in B$.  

Define a map $f : G \to G$ by
\[
f(g_i + b_j) = g_{\theta_{1, i,j}(i)} + b_{\theta_{2,i,j}(j)} + \widehat{a}_{i,j} + a_{\phi_{2, i,j}(j)},
\]
and extend it to all of $G$, for instance, by setting $f(g_i + b_j + a_t) = f(g_i + b_j)$.  
We claim that $f$ fulfills the hypotheses of the proposition.

First, modulo $A$ we have
\[
f(g_i + b_j) \equiv g_{\theta_{1, i,j}(i)} + b_{\theta_{2, i,j}(j)} \pmod A,
\]
and so,  by \eqref{eq:theta-bijec}, the elements $f(g_i + b_j)$ are pairwise distinct modulo $A$ for different pairs $(i,j)$.

To verify the condition modulo $B$, note that, using \eqref{eq:hat{a}}, we obtain
\[
f(g_i + b_j) - g_i
= g_{\theta_{1, i,j}(i)} + b_{\theta_{2, i,j}(j)} + \widehat{a}_{i,j} + a_{\phi_{2, i,j}(j)} - g_i
\equiv g_{\phi_{1, i,j}(i)} + a_{\phi_{2, i,j}(j)} \pmod B.
\]
Now \eqref{eq:phi-bijec} implies that
\[
f(g_i + b_j) - g_i \not\equiv f(g_{i'} + b_{j'}) - g_{i'} \pmod B
\]
whenever $(i,j) \neq (i',j')$. This completes the proof of the proposition.
\end{proof}

We remark that the proposition does not generalize to arbitrary abelian groups because it relies on the Hall-Paige Theorem applied to the  group  $H$ (as it appears in the proof). In particular if $H$ contains a cyclic Sylow 2-subgroup, which can occur when $|G|$ involves more than one prime the proof above fails and the proposition itself is not true. For instance, consider $G = \Z_2 \times \Z_3 \times \Z_3$, $A= \{ 0\} \times \Z_3 \times \{0\}$, and $B = \{0 \} \times \{ 0\} \times \Z_3$. Here $\{ g_i \mid i \leq 2 \} = \{ g_1=(0,0,0), g_2=(1, 0, 0)\}$ and the conclusion of the proposition states that there exists $f: G \to G$ so that the six elements  $f(g_i +b_j )$ are all distinct modulo $A$ and thus cover the 6 cosets of $A$ in $G$, where $B = \{ b_j \mid j \leq 3 \}$.  Similarly $f(g_i+b_j)-g_i$ cover all the 6 cosets of $B$ in $G$. If $f(g_i+b_j)= (x_{i,j}, y_{i,j}, z_{i,j}) \in G $ then looking at the cosets of $A$ in $G$ we conclude that exactly three elements (out of the 6 distinct elements $\{ f(g_i+b_j) \}_{i,j}$)  have first coordinate $x_{i,j}=0 $ and the other three have $x_{i,j}=1$. So 
\begin{equation}\label{eq:sum1}
\sum_{i,j}x_{i,j}=3 \equiv 1 \pmod2.
\end{equation}

Similarly,  for the elements $f(g_i+b_j)-g_i= (\hat{x}_{i,j}, \hat{y}_{i,j}, \hat{z}_{i,j})$ we have three of them having  their first coordinate $\hat{x}_{i,j}$  equal to  $0$ and the other three equal to $1$. So 
\begin{equation}\label{eq:sum2}
\sum_{i,j}\hat{x}_{i,j}=3 \equiv 1 \pmod2.
\end{equation}

Now notice that 
\[
(\hat{x}_{i,j}, \hat{y}_{i,j}, \hat{z}_{i,j})= f(g_i+b_j)-g_i= \begin{cases}
f(g_i+b_j)= (x_{i,j}, y_{i,j}, z_{i,j}), \, \, \text{ if }  i=1\\
f(g_i+b_j)-(1,0,0)= (x_{i,j}-1, y_{i,j}, z_{i,j}), \, \, \text{ if }  i=2.  
\end{cases}
\]
Hence, \eqref{eq:sum1} implies 
\[
\sum_{i,j}\hat{x}_{i,j}= \sum_{i,j} x_{i,j} -3 \equiv 0 \pmod2. 
\]
This clearly contradicts \eqref{eq:sum2}, proving that \autoref{prop:f-transversal} does  not extend to arbitrary finite abelian groups.

\section { Reductions and notation}\label{sec:Reductions}
This section presents initial reductions to the general problem and introduces the main players of the proofs that will follow.

\subsection{Reductions}

 The next two lemmata are well  known  (see for example  \cite{kol97} and \cite{ALS24})  and we only mention them  here.   The first one shows that it suffices to assume that  $G=X+Y+Z$ (for its proof one might see ~\cite[Lemma 2.2]{ALS24}). 
\begin{lemma}\label{Lem:G=X+Y+Z}
Let $R$ an abelian group (not necessarily finite) and $K, L, M $ subgroups of $R$ then     $\T_R(K, L, M) =\emptyset $ if and only if $\T_{K+L+M}(K, L, M) = \emptyset$.  
\end{lemma}
The next result enables us to work with quotient groups (see \cite[Lemma 2.3]{ALS24}). 
\begin{lemma}\label{lem:factor} 
Let $R$ be an  abelian group and  $K, L, M \leq R$  with $N\leq K \cap L \cap M$. Then 
$\T_R(K, L, M)\ne\emptyset$ if and only if $\T_{R/N}(K/N, L/N,  M/N) \ne\emptyset $. 
In particular, 
$\{g_1, \ldots, g_m\} \in \T_R(K, L, M )$ if and only if $\{g_1+N, \ldots, g_m+N\} \in \T_{R/N}(K/N, L/N, M/N)$.

\end{lemma}
An immediate consequence of the above lemma is that we can always assume that $X \cap Y \cap Z = \{0\}$. 

The first simple but interesting result in the direction of our main theorem is  the special case where the (common) index of $X$ (and $Y$, $Z$)  in $G$ is a prime $p$.  
\begin{lemma}\label{Lem:Index-p}
Assume $G$ is finite abelian $p$-group and $X, Y, Z$ subgroups of index $p$ and trivial intersection.  Then $T_G(X, Y, Z) = \emptyset$ if and only if $p=2$ and $G $ is isomorphic to the Klein group $V_4$. 
\end{lemma}

\begin{proof}
 Clearly the Klein group $V_4$ satisfies the lemma. For the other direction assume first that  $G \neq X \cup Y \cup Z$.  Then any $g \in G \setminus (X \cup Y \cup Z)$ generates a common transversal $\{ i \, g \mid 1 \leq i \leq p \}$ for $X, Y$, and $Z$. If $G= X \cup Y \cup Z$ then  it is well known  (see, e.g., \cite{gaetano,Bruckheimer01011970}) that $G$ has a factor group isomorphic to  $V_4$, so  $p=2$. Moreover,  $X, Y, Z$ are distinct maximal subgroups of $G$, hence  $G = X+Y= X+Z=Y+Z$, implying $|X \cap Y| = |X \cap Z| = |Y \cap Z| =|G|/4$. Since $|X\cap Y \cap Z|= 1$, the inclusion-exclusion principle yields 
   \[
|G|= |X\cup Y \cup Z|= \frac{3|G|}{2} - \frac{3|G|}{4} + 1 = \frac{3|G|}{4} + 1.
    \]
    Therefore  $|G|= 4$ and so  $G \cong V_4$,   as desired.  
\end{proof}
    
\subsection{Notation}

Let us now introduce the  notation that will be used henceforth.   $G$ is always  assumed to be a finite abelian $p$-group with subgroups $X, Y, Z$ of the same order  $m$ that intersect trivially, i.e., $X \cap Y \cap Z = \{0\}$. We also assume that $G= X+Y+Z$. 
We write  $\hat{X}= X \cap (Y+Z)$, $\hat{Y}= Y \cap (X+Z)$, and $\hat{Z}= Z \cap (X+Y)$ and assume that their orders are $l_1, l_2$, and $l_3$, respectively. 
Moreover we write 
$A:= X\cap Y$, $B := X\cap Z$ and $C := Y \cap Z$, and their orders as $r_1, r_2$ and $r_3$ respectively.  These subgroups are central to the proofs that follow and we collect their definitions, orders and relations, for easy  reference, below;   
\begin{align}\label{eq:def.ABC}
\hat{X}= X \cap (Y+Z), \, \, \hat{Y}&=Y \cap (X+Z), \, \,  \hat{Z}= Z \cap (X+Y)\\
l_1= |\hat{X}|,\, \,   l_2 &=|\hat{Y}|, \, \,  l_3= |\hat{Z}| \\
A= X \cap Y, \,\, \,   B&= X\cap Z, \, \, \,  C= Y \cap Z\\
r_1=|A|, \, \, r_2&=|B|, \, \, r_3=|C|.  \\
\end{align}
We also define 
\begin{align}\label{eq:defT}
   T_X:= A \oplus B &\leq \hat{X} \leq  X\\
    T_Y:=A \oplus C &\leq \hat{Y} \leq  Y\\
     T_Z:=B\oplus C &\leq \hat{Z} \leq  Z   
\end{align}
So $l_i \mid m$ for all $i$,  while $r_1 r_2 \mid l_1$, $r_1r_3 \mid l_2$ and $r_2 r_3 \mid l_3$.
Furthermore, the sum of $A=(X\cap Y),  B=(X\cap Z)$ and $C=(Y \cap Z)$ is easily checked to be direct and we denote it as 
\begin{equation}\label{eq.T}
T = A\oplus B \oplus C.
\end{equation}
Let us also point that any of the three "hat"-groups is contained within the sum of the other two, i.e., 
\begin{equation}\label{eq:hat-groups}
\hat{X} + \hat{Y}= \hat{X} +\hat{Z}= \hat{Y} + \hat {Z}, 
\end{equation}
because any relation $x+y=z$ with $x\in X, y\in Y$, and $z \in Z$ implies $x \in \hat{X}, y\in \hat{Y}$, and $z \in \hat{Z}$. It is also straightforward that 
\begin{equation}\label{eq:hat-intersections}
  \hat{X} \cap \hat{Y} = A, \, \, \hat{X} \cap \hat{Z}= B, \, \, \hat{Y} \cap \hat{Z}= C.   
\end{equation}

Because  $G = X+Y+Z$,   we can compute the order of $G$ as a function of $m, l_i$ and $r_i$ as 
\begin{equation}\label{eq:orderG}
|G|= |X+Y+Z|= \frac{m^3}{l_1r_3}=\frac{m^3}{l_2r_2}= \frac{m^3}{l_3r_1}
\end{equation}
and thus 
\begin{equation}\label{eq:lr}
l_1r_3=l_2r_2=l_3r_1.
\end{equation}
This in turn implies 
\begin{equation}\label{eq:[hatY:T_Y]}
[\hat{X}:T_X]=[\hat{Y}:T_Y]= [\hat{Z}:T_Z]=\frac{l_1}{r_1 r_2}= \frac{l_2}{r_1 r_3}=\frac{l_3}{r_2 r_3}.
\end{equation}
The relationship between the "hat" groups (see \eqref{eq:hat-groups}) allows us to simultaneously choose transversals for $T_X, T_Y,$ and $T_Z$ in $\hat{X}, \hat{Y},$ and $\hat{Z}$, respectively. We prove this in the next lemma.  
\begin{lemma}\label{lem:correspondence-Transversals}
Let $\{\hat{z_i} \mid i \leq \frac{l_1}{r_1 r_2}\}$ be a transversal of $T_Z$ in $\hat{Z}$. Then $\hat{z_i} = \hat{x}_i +\hat{y}_i$ for some $\hat{x}_i\in X$ and $\hat{y}_i \in Y$, where $i\le \frac{l_1}{r_1r_2}$. Moreover, $\{\hat{x}_i \mid i \leq  \frac{l_1}{r_1r_2} \}$ and $\{\hat{y}_i \mid i \leq  \frac{l_1}{r_1r_2}\}$ are transversals of $T_X$ in $\hat{X}$ and $T_Y$ in $\hat{Y}$, respectively.
\end{lemma}

\begin{proof}
As $\hat{Z} \leq \hat{X} +\hat{Y}$, we can express each $\hat{z}_i$ for $i \leq \frac{l_1}{r_1r_2}$ as $\hat{z_i} = \hat{x}_i +\hat{y}_i$, where $\hat{x}_i\in X$ and $\hat{y}_i \in Y$. Thus,  it remains to show that $\{\hat{x}_i\}$ and $\{\hat{y}_i\}$ are inequivalent modulo $T_X$ and $T_Y$, respectively. As the problem is symmetric for $\hat{X}$ and $\hat{Y}$, is suffices to prove $\{\hat{x}_i\}$ is a transversal of $T_X$ in $\hat{X}$. 

Assume that $x_i - x_{i'} \in T_X$ for some $i,i'\le \frac{l_1}{r_1r_2}$. Then $z_i - z_{i'} = x_i-x_{i'} +y_i - y_{i'} \in \hat{Z} \cap (T_X + \hat{Y})$.  Let us now observe
\begin{multline}
\hat{Z} \cap (T_X + \hat{Y})= \hat{Z} \cap (A\oplus B + \hat{Y})= \hat{Z} \cap (B \oplus C + \hat{Y})\\
=B \oplus C +(\hat{Z} \cap \hat{Y}) = B \oplus C +C = B\oplus C= T_Z, 
\end{multline}
where the second equality uses $A, C \subseteq \hat{Y}$, the third follows from  Dedekind's modular law since $B\oplus C \subseteq \hat{Z}$, and the fourth uses \eqref{eq:hat-intersections}. Thus, $x_i \equiv x_{i'} \pmod {T_X}$ implies $z_i \equiv z_{i'} \pmod {T_Z}$. Since $\{z_i\}$ is a transversal, we conclude that $\{ x_i\}$  are all inequivalent modulo $T_X$ and so it is a transversal for $T_x$ in $\hat{X}$,  because it has the right order.
   This completes the proof of the lemma. 
\end{proof}

\section{ The case $G = X \oplus Y $}\label{sec:DirectSum}
This section addresses the special case where $G$ is a finite abelian  $p$ group and is  the direct sum of two subgroups, $X$ and $Y$. This case is crucial because (as we will see) is the only  instance of three subgroups lacking a common transversal. Our first theorem states:

\begin{theorem} \label{Thm:Direct-Sum-of-Two}
Assume $G$ is an abelian $p$-group. Let $X, Y, Z$ be proper subgroups of the same index in $G$ such that $G= X\oplus Y$. Then $\T_G(X, Y, Z) = \emptyset$ if and only if $p = 2$ the groups  $X, Y, Z$ are all cyclic and $G = X \oplus Y = X \oplus Z = Y \oplus Z$.  
\end{theorem}

\begin{proof}
In our notation (see \eqref{eq:def.ABC},   \eqref{eq:lr}), we have  $r_1=1$, $\hat{Z} = Z$, $l_3= m=|X|$ and $|G|=|X\oplus Y|= m^2$. Because  $G$ is the direct sum of $X$ and $Y$ every element of $Z$ can be uniquely written as an element of $X$ plus an element of $Y$, so that 
\begin{equation}\label{eq:Z}
    Z=\{z_i=x_i +y_i \: | \: i\le m\}
\end{equation} 
for some $x_i \in X$ and $y_i\in Y$ not necessarily distinct. Now we distinguish cases, depending on the values of $r_2$ and $r_3$. 

 \textbf{ Case 1.} ($r_1=r_2=r_3=1$).
 
 In this case all  $A, B, C$ are  trivial groups.  So  $l_1=l_2=l_3 = m$ (by \eqref{eq:lr}) and thus, not only $\hat{Z}=Z$, but also  $\hat{X}=X$ and   $\hat{Y}=Y$. Furthermore, $G = X\oplus Y = X \oplus Z= Y \oplus Z$  and all the groups $X, Y, Z$ are isomorphic; 
    \begin{equation*}
        Y\cong (X\oplus Y)/X = (X\oplus Z)/X \cong Z\cong (Y\oplus Z)/Y =(Y\oplus X)/Y \cong X
    \end{equation*}
    Furthermore, \autoref{lem:correspondence-Transversals} implies that the elements $\{z_i\}$ of $Z$ can be expressed as $\{z_i = x_i + y_i \mid i \le m\}$, where $\{x_i \mid i \leq m\} = X$ and $\{y_i \mid i \leq m\} = Y$. We fix this enumeration for the elements of $X$, $Y$, and $Z$.

 We first claim that a set $F$ of cardinality $m$ is a common transversal of $X$ and $Y$ if and only if there exists a bijection $\phi: Y \to Y$ so that 
\begin{equation}\label{eq: F}
    F=\left \{x_i + \phi (y_i) \: | \: i\le m\right \}
    \end{equation} 
 The "only if" direction is clear; we now prove the "if" direction.
Let $F'$ be a common transversal of $X$ and $Y$ in $G$. As $F'$ is a transversal of $X$ in $G$, we can enumerate its elements as $f_i$, $i \le m$, such that $f_i \equiv x_i \pmod{Y}$, implying $f_i = x_i + y^{i}$ for some $y^{i} \in Y$, for all $i \le m$. Because $F'$ is also a transversal of $Y$ in $G$, we have $\{y^{i} \mid i \le m\} = Y$. Defining $\phi : Y \to Y$ such that $\phi(y_i) = y^{i}$, we observe that $\phi$ is a bijection of $Y$ and  $F'$ has the form given by equation $\eqref{eq: F}$.

  Now, a  transversal of $X, Y$ in $G$, that we may assume to be of the form $\eqref{eq: F}$ for some bijection $\phi$, is a transversal of $Z$ in $G$ if and only if 
   \[
   x_i + \phi(y_i) \not \equiv x_{i'} + \phi(y_{i'}) \pmod{Z}
   \]
   for all distinct $i, i' \le m$.
   
Since $x_i \equiv -y_i \pmod{Z}$ (because $z_i = x_i + y_i$), a common transversal $F=\{x_i + \phi (y_i) \: | \: i\le m\}$ of $X$ and $Y$ in $G$ is also a transversal of $Z$ if and only if $-y_i +\phi(y_i) - (-y_{i'} +\phi(y_{i'}) ) \notin Z$ for all $i \neq i' \leq m$. This is equivalent to requiring that $\phi$ apart from being a bijection of $Y$ should also satisfy  
\[
\phi(y_i) - y_i\neq \phi(y_{i'}) - y_{i'}
\]
for all $i\neq i'\le m$. In other words, the map $\psi: Y \to Y$ defined as
    \begin{equation*}
        \psi(y_i)= \phi (y_i) - y_i
    \end{equation*}
    is a bijection on $Y$ and simultaneously $\psi(y_i) +y_i= \phi(y_i)$ is also a bijection on $Y$,   i.e, $\psi$ must be a complete map in $Y$. By the Hall-Paige Theorem for finite abelian groups, a complete map exists if and only if $Y$ (and so $X$ and $Z$ as they are isomorphic) is not a cyclic $2$-group. Therefore, a common transversal for $X, Y, Z$ in $G=X\oplus Y= X \oplus Z= Y \oplus Z$ exists if and only if there exists a complete map $\psi:Y\to Y$, which holds if and only if $p$ is an odd prime or $p=2$ and $Y$ is not cyclic.

This completes Case 1 of our theorem. Thus,  may assume that not all $r_i$ are trivial. Observe though that  to complete the proof, it suffices to show that a common transversal always exists in the remaining cases. 
  Without loss we may assume that  $r_3=|C|=|Y\cap Z|>1$. 
  
  {\bf Case 2. } ($r_1=1, r_3 >1$) 
  
By \autoref{lem:correspondence-Transversals}, we can select a transversal $\{\hat{z_j} \mid j \leq \frac{l_1}{ r_2}\}$ of $T_Z$ in $\hat{Z}$ such that 
\[
\hat{z_j} = \hat{x}_j +\hat{y}_j,
\]
where 
$\{ \hat{x}_j\mid  j \leq \frac{l_1}{ r_2} \}$ and $\{ \hat{y}_j \mid j \leq \frac{l_1}{ r_2} \} $ are transversals of $T_X$ in $\hat{ X} $ and $T_Y$ in $\hat{Y}$, respectively. Here, $T_X= B$, $T_Y= C$ and $T_Z= B \oplus C$, where $B$ may be trivial, in which case $\{ \hat{x}_j \}= \hat{X}$. Thus,
\[
\hat{X} = \bigsqcup_{j=1}^{l_1/r_2} \hat{x}_j + B , \, \,\, \,  \hat{Y} = \bigsqcup_{j=1}^{l_1/r_2} \hat{y}_j + C, \, \, \, \, Z= \hat{Z} = \bigsqcup_{j=1}^{l_1/r_2} \hat{z}_j + (B \oplus C).  
\]
Fix an arbitrary enumeration of the elements of $B=\{b_i \: |\: i\le r_2\}$ and $C=\{c_t \: | \: t\le r_3\}$. Let us note that in our case $m = l_3 =  l_1r_3= l_2 r_2$, hence   $ [X:\hat{X}]=r_3$ and  
$ [Y:\hat{Y}]=r_2$.   We also fix  transversals $\{ x_i   \mid i  \leq r_3 \} $  and $\{y_i  \mid i \leq r_2\}$ of   $\hat{X} $ in $X$ and $\hat{Y}$ in $Y$ respectively.   
\begin{figure}[h]
\includegraphics[width=0.5\textwidth]{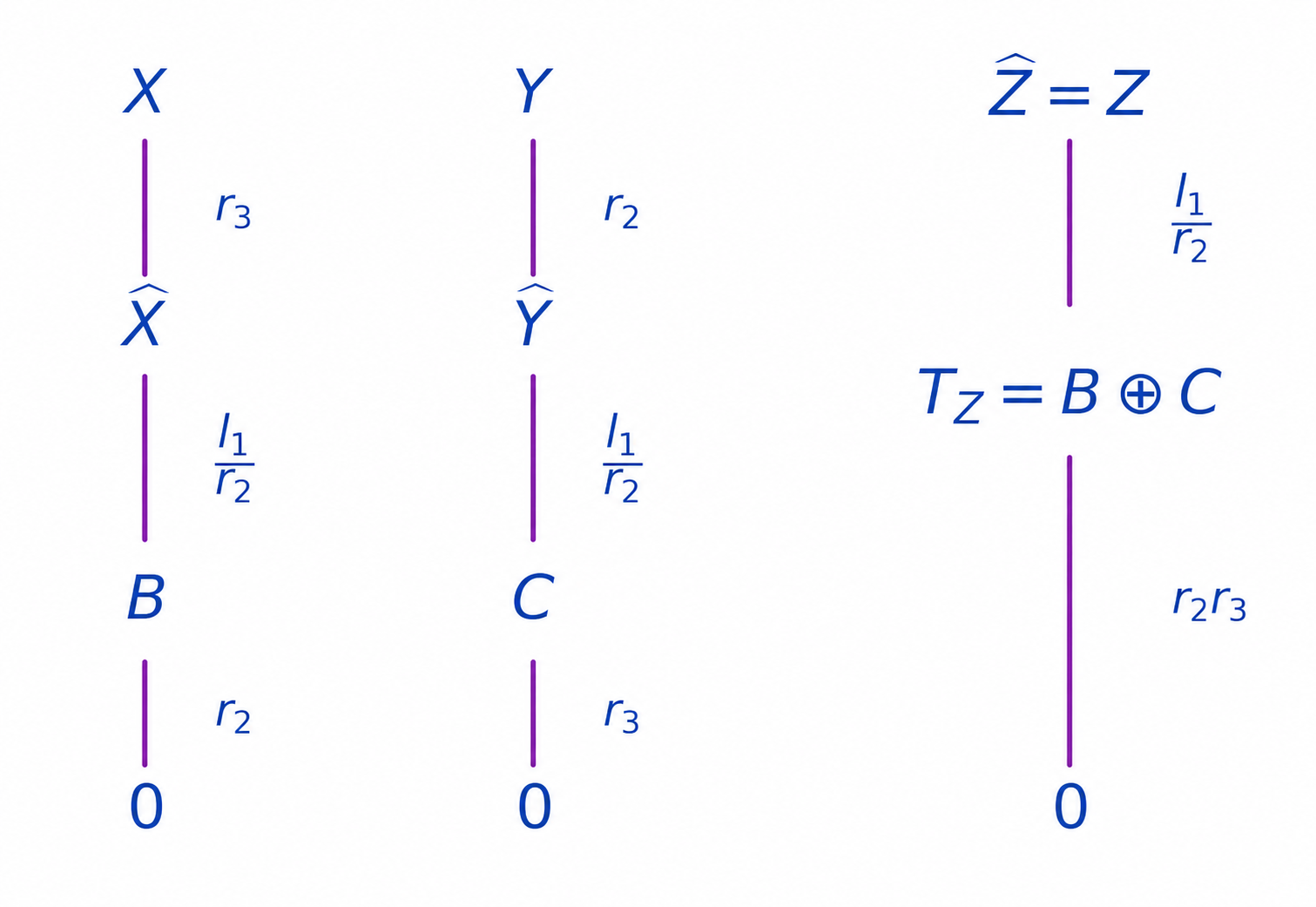}
\caption{$X,Y,Z$ along with their  involved subgroups}
\label{fig:myfigure2}
\end{figure}

We are now applying \autoref{lem:bijectionInQuotient} to the groups $Y, \hat{Y}$ and $C$ in the 
place of $G, B$  and $A$ respectively. This yields  a bijection $f: Y \to Y $  such that for any fixed $t \leq r_3$, 
\[
f(y_i + \hat{y}_j +c_t) - \hat{y}_j \not \equiv f(y_{i'} + \hat{y}_{j'} +c_{t}) - \hat{y}_{j'} \pmod {C}
\] 
whenever $(i, j) \neq (i',j')$.  We claim that the set 
    \begin{equation}\label{eq:F}
        F:= \left\{ x_t + \hat{x}_j + b_i  + f(y_i +\hat{y}_j +c_t) \: | \: i\le r_2, j\le \frac{l_1}{r_2}, t\le r_3\right \}
    \end{equation}
   is a common transversal for $X, Y, Z$ in $G$. 
   
   Clearly it is of the correct size $l_1r_3 = m= [G:X]$. Since $f$ is a bijection of $Y$, taking the elements of $F$ modulo $X$ yields every element of $Y$, giving a transversal of $X$ in $G = X \oplus Y$. Similarly, taking the elements of $F$  modulo $Y$ yields every element of $X$ because every element in $X$ can be written, in a unique way, as $x_t + \hat{x}_j + b_i$ for appropriate indices. Thus $F$ modulo $Y$ gives  a transversal of $Y$ in $G$. It remains to check that the same  holds modulo $Z$.
   But now notice that $F$ modulo $Z$ reduces to the set 
   \[
\left\{ x_t - \hat{y}_j + f(y_i +\hat{y}_j +c_t) \: | \: i\le r_2, j\le \frac{l_1}{r_2}, t\le r_3\right \}
   \]
   because $B \leq Z$ and $\hat{x}_j +\hat{y}_j = \hat{z}_j \in Z$. We aim to show that all  elements in the above set lie in distinct $Z$-cosets. Assume the contrary and let  $x_t - \hat{y}_j + f(y_i +\hat{y}_j +c_t ) =   x_{t'} - \hat{y}_{j'} + f(y_{i'} +\hat{y}_{j'} +c_{t'}) +z $ for some $z \in Z$ and some $(i, j, t) \neq (i', j', t')$. Then $x_t -x_{t'}  \in X \cap (Y+Z) = \hat{X} $,  forcing $t=t'$. Hence $ - \hat{y}_j + f(y_i +\hat{y}_j +c_t ) =  - \hat{y}_{j'} + f(y_{i'} +\hat{y}_{j'} +c_{t}) +z$ and thus $z \in Z \cap Y= C$. So we end up with  $ f(y_i +\hat{y}_j +c_t)- \hat{y}_j \equiv  f(y_{i'} +\hat{y}_{j'} +c_{t})    - \hat{y}_{j'}  \pmod{C}$ for some $(i, j) \neq (i', j')$, contradicting the choice of $f$. 
   
   This completes the proof of Case 2 and the proof of the theorem.

\end{proof}

\section{When G  = X+Y  }\label{sec:sumOfTwo}
This section addresses the case  the $p$-group$G$ is given as   $G = X+Y$, not necessarily as a direct sum, and shows that it reduces to the direct sum case covered in \autoref{Thm:Direct-Sum-of-Two}. The proof of \autoref{Thm:Sum-of-Two}, the section's main theorem, uses induction on the subgroup index. Let us remark that while the induction proceeds smoothly for odd order groups, the prime $p=2$ requires a special case verification, handled in the proposition below. Although the proposition is  stated for arbitrary prime $p$ with an identical proof regardless of parity, 
we remind our reader that a simplified proof for $p$ odd exists where a "good" function $f$ could be defined as $f(x)=-x$. But for $p=2$, existence of a "good" $f$ relies on \autoref{prop:f-transversal}.   
 
\begin{proposition}\label{prop.missing-X+Y}
Let \(G\) be a finite abelian \(p\)-group of order \(|G|=n^2r^3\), and let \(X, Y, Z\) be subgroups of \(G\) of order \(|X|=nr^2\) such that \(X\cap Y\cap Z = \{0\}\) and \(G = X + Y\).  Assume further that  $1< r=|A|=|B| $ and  $|C| \geq r$, where $A= X \cap Y, B= X \cap Z$ and $C= Y \cap Z$. Let  $C_1 \leq C$ of order $r$.  If  $X/ (A+B), Y / (A+C_1)$ and $Z/ (B+C_1)$ are cyclic groups of order $n$  then the subgroups \(X, Y, Z\) have a common transversal in \(G\).
\end{proposition}

\begin{proof}
With our notation, $r_1=r_2=r$ and $T_X= A \oplus B$ has order $r^2$, where $l_2=l_3= nr^2$ and $l_1 r_3 = nr^3$ (see \eqref{eq:lr} for the relations of $l_i, r_i$ ). Hence $\hat{Y} = Y $ and $\hat {Z}= Z$ while $\hat{X}$ has index $\frac{nr^2}{l_1} =\frac{r_3}{r}$ in $X$.   Furthermore the  groups $\hat{X}/ T_X, \hat{Y}/T_Y $ and  $\hat{Z}/T_Z, $  are all cyclic by hypothesis (as quotients of cyclic groups), and they all have order $\frac{l_1}{r^2}= \frac{nr }{r_3}$. The situation is described in \autoref{fig:myfigureProp2}.
\begin{figure}[h]
\includegraphics[width=0.6\textwidth]{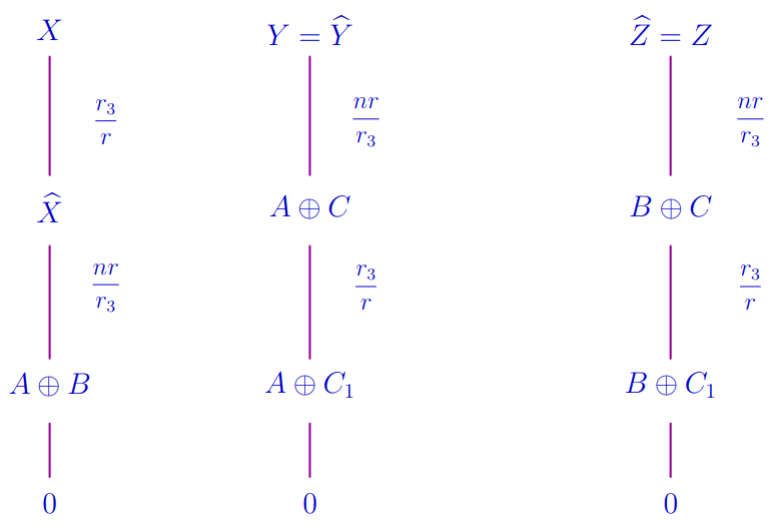}
\caption{$X,Y,Z$ along with their  involved subgroups}
\label{fig:myfigureProp2}
\end{figure}

In view of \autoref{lem:correspondence-Transversals}, we can pick  $\hat{z}_0 \in \hat{Z}$  with $\hat{z}_0 +T_Z$ being a generator of the quotient group $\hat{Z}/ T_Z$, so that    
\begin{equation}\label{eq:z_0=x_0+y_0}
\hat{z}_0 = \hat{x}_0 + \hat{y}_0
\end{equation}
for  $\hat{x}_0 \in \hat{X}$ and $\hat{y}_0 \in \hat{Y}$. 
Moreover,   $\hat{x}_0 +T_X$ and $\hat{y}_0 +T_Y$ are generators of $\hat{X}/T_X $ and $\hat{Y}/T_Y $ respectively. 
Hence for all $ 1 \leq j \leq \frac{nr }{r_3}$ the sets $\{ j \, \hat{x}_0 \} , \{ j \, \hat{y}_0 \}$ and $\{ j\, \hat{z}_0 \}$ are transversals of $T_X, T_Y, T_Z $ in $\hat{X}, \hat{Y}$ and $\hat{Z}$ respectively. 
Let $\{d_t\}_{ t \leq r_3/r}$ and $\{x_t\}_{t \leq r_3/r}$ be transversals of $C_1$ in $C$ and $\hat{X}$ in $X$, respectively. So 
\[
    \hat{X} = \bigsqcup_{j \leq \frac{nr }{r_3}} j \, \hat{x}_0 + T_X, \, \, 
    \hat{Y} = \bigsqcup_{j \leq \frac{nr }{r_3}} j \, \hat{y}_0 + T_Y, \, \, \, 
    \hat{Z} = \bigsqcup_{j \leq \frac{nr }{r_3}} j \, \hat{z}_0 + T_Z
    \]
    and 
    \[
    C= \bigsqcup_{t \leq \frac{r_3 }{r} } d_t + C_1, \, \, \, 
    X= \bigsqcup_{t \leq \frac{r_3 }{r} } x_t + \hat{X}.
    \]
Next we apply \autoref{prop:f-transversal} to the group $\hat{X} $ and its subgroups $A$ and $ B= \{ b_i \mid  i \leq r \}$. This yields a function $f : \hat{X} \to \hat{X}$ satisfying  
\begin{align}
f(j \, \hat{x}_0 + b_i )  &\not \equiv  f(j' \, \hat{x}_0 + b_{i'} ) \pmod A  \\
f(j \, \hat{x}_0 + b_i ) - j \,\hat{x}_0  &\not \equiv  f(j' \, \hat{x}_0 + b_{i'} )  - j' \,\hat{x}_0 \pmod B
\end{align}
for all $(i, j ) \neq (i', j') $. 
Let us  write $C_1= \{ c_i \mid i \leq  r\}$ and  consider the set 
\[
F= \left\{ j \, \hat{y}_0 + c_i +f(j \, \hat{x}_0 +b_i) + d_t + x_t \mid  j \leq \frac{nr }{r_3}, \, i \leq r, \, t \leq r_3/r \right\} 
\]
 Clearly $|F|= nr = [G:X]$. We claim that $F$ is the desired common transversal of $X, Y, Z$. 

We first observe that $F \pmod{X}$ simplifies to the set $\{ j \, \hat{y}_0 + c_i + d_t \}$ for valid indices $i, j$ and $t$.  If $j\, \hat{y}_0 + c_i + d_t \equiv j'\, \hat{y}_0 + c_{i'} + d_{t'} \pmod X$ then 
$(j -j') \, \hat{y}_0 \in Y \cap (C+X) = C + (Y \cap X ) = C+A=T_Y$, so $j = j'$. Therefore, $d_t - d_{t'} \in C \cap (C_1 + X) = C_1$, implying $t = t'$. Consequently, $c_i - c_{i'} \in C_1 \cap X = \{0\}$, demonstrating that all elements of $F$ are inequivalent modulo $X$.

Similarly  $F$ modulo $Y$ reduces to $\{ f(j \, \hat{x}_0 +b_i) +x_t \}$. If 
$ f(j \, \hat{x}_0 +b_i) +x_t \equiv f(j' \, \hat{x}_0 +b_{i'}) +x_{t'} \pmod Y $ then the first indices that coincide are $t=t'$ owing to $x_t -x_{t'}  \in X \cap (\hat{X} +Y) = \hat{X}$. Thus $f(j \, \hat{x}_0 + b_i )  \equiv  f(j' \, \hat{x}_0 + b_{i'} ) \pmod Y $. But  $f$ maps to $\hat{X}$ and $\hat{X} \cap Y = A$, making the last equivalence modulo $Y$ an equivalence modulo $A$. The definition of $f$ makes $(i, j) = (i', j')$ as required. 

Finally $F$ modulo $Z$ simplifies to $\{ j \, \hat{y}_0 + f(j \, \hat{x}_0 +b_i) +x_t \} \equiv  \{ -j \, \hat{x}_0 + f(j \, \hat{x}_0 +b_i) +x_t  \} \pmod {Z}$,  by \eqref{eq:z_0=x_0+y_0}. 
As previously, if  $f(j \, \hat{x}_0 +b_i) - j \, \hat{x}_0  +x_t \equiv f(j' \, \hat{x}_0 +b_{i'})- j' \, \hat{x}_0 +x_{t'} \pmod Z$, then $t = t'$ because $x_t -x_{t'} \in X \cap (\hat{X} + Z) = \hat{X}$. 
So $f(j \, \hat{x}_0 +b_i) - j \, \hat{x}_0  \equiv f(j' \, \hat{x}_0 +b_{i'})- j' \, \hat{x}_0  \pmod Z$. But $Z\cap \hat{X} = B$ and thus $f(j \, \hat{x}_0 +b_i) - j \, \hat{x}_0  \equiv f(j' \, \hat{x}_0 +b_{i'})- j' \, \hat{x}_0  \pmod B$ which in turn implies $(i, j)=(i', j')$,  by the definition of $f$. This shows that all elements of $F$ are inequivalent modulo $Z$, and the proposition follows.
\end{proof}

We are now ready to prove our next  theorem. Let us observe that the following  theorem  along with \autoref{Thm:Direct-Sum-of-Two},  implies  that if $G=X+Y$ is a $p$-group, that   is not a direct sum of two of the subgroups $X, Y, Z$, then a common transversal always exists. 

\begin{theorem}\label{Thm:Sum-of-Two}
 Let $X, Y, Z$ be proper subgroups of an abelian $p$-group $G$ with equal index such that $G = X + Y$ and $X \cap Y \cap Z = \{0\}$. Then $\T_G(X, Y, Z) = \emptyset$ if and only if $p=2$, the groups $X, Y, Z$ are cyclic and  $G = X \oplus Y = X \oplus Z = Y \oplus Z$.
\end{theorem}

\begin{proofof}
If $X, Y, Z$ are cyclic, the result follows from Theorem B in~\cite{ALS24}. Otherwise, assume $X, Y, Z$ are not all cyclic. We will show that (in this non-cyclic case) $\T_G(X, Y, Z) \neq \emptyset$, from which the theorem follows.

We proceed by induction on $[G:X]$, with the base case $[G:X]=p$ covered in \autoref{Lem:Index-p} initiating the induction. 

Since $Z \leq X+Y$, we have $\hat{Z}= Z$, thus $l_3= m$, and \eqref{eq:lr} simplifies to
\[
l_1r_3= l_2r_2=mr_1.
\]
Also $l_3 \geq l_1, l_2$ as $l_i \leq m$ and thus $r_1 \leq r_2, r_3$ and $r_1^2 \leq r_1 r_2 \leq l_1  \leq m$.  In addition, equation \eqref{eq:orderG}  implies   $|G|= \frac{m^2}{r_1}$ and the index of $X$  in $G$ is $\frac{m}{r_1}$. At this point,  where the index $[G:X]$ is fixed, let us state the inductive hypothesis for clarity; assume $H= H_1+H_2$ is an abelian $p$-group and  $H_1, H_2, H_3$ are  subgroups of the same order, that   intersect trivially and are not all cyclic, then they share a common transversal in $H$ as long as their index is  strictly less than  $\frac{m}{r_1}$. 

Because $G$ is a $p$-group, the fact that $r_1 \leq r_2, r_3$ ensures the existence of subgroups $L \leq B$ and $K \leq C $  of order $r_1$. Then  
\begin{align}
|X+K|&= |X \oplus K| = mr_1\\
|Y+L|&= |Y \oplus L| = mr_1\\
|Z+ A |&= |Z \oplus A |= mr_1. 
\end{align}
The group $K+L+A$ is in fact a direct sum that we denote by  $N := K \oplus L \oplus A$. This can be verified by noting, for example, that $K \cap (L+A) \leq C \cap (B +A) \leq C \cap X = \{0\}$ (the other intersections are similar). Moreover, $|N|= r_1^3$, and \autoref{fig:myfigure3} illustrates the selected subgroups. 

 \begin{figure}[h]
\includegraphics[width=0.5\textwidth]{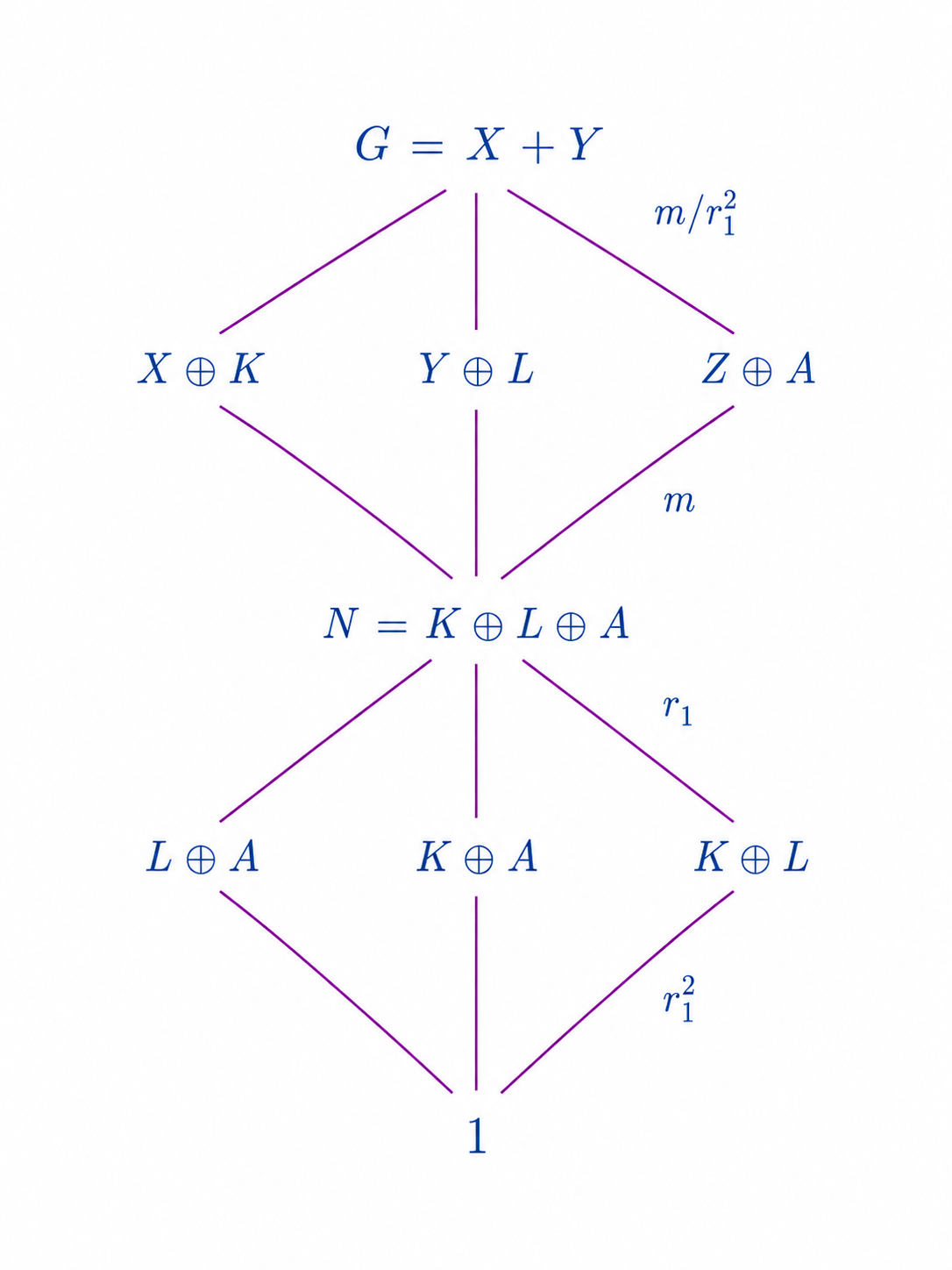}
\caption{The groups involved in this case}
\label{fig:myfigure3}
\end{figure}

We next distinguish three cases for $r_1$. 

{\bf Case 1.} ($r_1 \neq 1$ and $r_1^2 \neq m$): \\
In this case, we can apply the inductive hypothesis twice to get a common transversal for $X, Y, Z$. 

On the one hand, the subgroups  $X\oplus K, Y\oplus L$ and $Z \oplus A $ are 
proper in $G$ (since $r_1^2 \neq m$) of the same order $mr_1$ and   $G$ remains the sum of two of them. In addition, their index is   $\frac{m}{r_1^2} $. 
We can easily verify that 
\[
N= K \oplus L \oplus A= (X\oplus K) \cap ( Y\oplus L)  \cap  (Z \oplus A).
\]
Consider the quotient group $\widebar{G}= G /N$ and its subgroups $\widebar{X}:= (X\oplus K) / N , \widebar{Y} : = (Y\oplus L)/N$, and $\widebar{Z} := (Z \oplus A) /N$, each with index $\frac{m}{r_1^2}$ in $\widebar{G}$. Clearly $1< \frac{m}{r_1^2} < \frac{m}{r_1}$, as $r_1^2 < m$ and  $1< r_1$. If $\widebar{X}, \, \widebar{Y}, \, \widebar{Z}$  are not all cyclic, the inductive hypothesis applies and yields a common transversal in $\widebar{G}$. If, on the other hand, $\widebar{X},\, \widebar{Y},\, \widebar{Z}$ are all cyclic, then Theorem B of \cite{ALS24} applies to this smaller triple. Therefore,  $\widebar{X}, \widebar{Y},$ and $\widebar{Z}$ do not have a common transversal in $\widebar{G}$ only if all are cyclic $2$-groups  and    
\begin{equation}\label{eq:widebarG}
\widebar{G}= \widebar{X} \oplus \widebar{Y}  = \widebar{X} \oplus \widebar{Z}=\widebar{Y} \oplus \widebar{Z}.
\end{equation}
Now 
\[
\widebar {Z} =\frac{Z\oplus A}{N}= 
\frac{Z+N}{N} \cong \frac{Z}{Z \cap N} = \frac{Z}{L\oplus K}, 
\]
and thus $Z/ (L \oplus K)$ is a cyclic group. So  $T_Z/(L \oplus K) = (B \oplus C)/(L\oplus K) \leq Z/(L \oplus K)$ should be a cyclic group, forcing $r_1=r_2$ or $r_1=r_3$. Without loss  we may assume $r_1=r_2$  and thus $L=B$ and $l_2 = l_3=m$ (by \eqref{eq:lr}). If $\widebar{Z} \cong  \Z_{2^k} $ for some $k \geq 1$, then $m =|Z|= |\widebar{Z}|r_1^2=  2^k r_1^2$ and $|G|= 2^{2k} r_1^3$. 

But now we are in the situation of \autoref{prop.missing-X+Y} (with $n=2^k$, $r=r_1 >1$, and $C_1=K$) and a   common transversal for $X, Y, Z$ in $G$ exists, proving  our induction. Thus, we may assume that \eqref{eq:widebarG} does not occur in which case $\T_{\widebar{G}}(\widebar{X}, \widebar{Y}, \widebar{Z})  \neq  \emptyset$ and thus  $\T_G(X\oplus K, Y \oplus L, Z \oplus (X \cap Y) ) \neq \emptyset$,   by \autoref{lem:factor}.  Let 
$\{ g_i \} _{i \leq \frac{m}{r_1^2}} $  be a common transversal of $X\oplus K, Y\oplus L$ and $Z \oplus A $ in $G$, that is  
\begin{equation}\label{eq:Sum2-g}
g_i \not\equiv g_{i'}  \bmod {X+K},  \, \, \bmod {Y+L}, \, \,  \bmod {Z+A}, 
\end{equation}
  for all   $ 1 \leq i \neq i' \leq m/ r_1^2 $.
  
On the other hand, the groups $K\oplus A,\, K \oplus L, \, L \oplus A$ intersect trivially,  have the same order $r_1^2 $ and  are proper noncyclic subgroups of $N$ (since $r_1 \neq 1$). Their index in $N$ is $r_1 < m/r_1$ (since $r_1^2 \neq m$), and $N= (K\oplus A) + (K \oplus L) $. Therefore, the inductive hypothesis applies, producing a common transversal $\{ w_j \}_{j \leq r_1}$ of $K\oplus A,\, K \oplus L$, and $ L \oplus A$ in $N$, that is,
\begin{equation}\label{eq:Sum2-w}
w_j  \not\equiv w_{j'}  \bmod {K \oplus A },  \, \, \bmod {K \oplus L}, \, \,  \bmod {L \oplus A}, 
\end{equation}
for all  $ 1 \leq j  \neq j' \leq r_1$.

We claim that 
\[
F = \left\{ g_i + w_j \mid i \leq m/r_1^2, \, j \leq r_1 \right\},
\]
is a common transversal of $X$, $Y$, and $Z$ in $G$. Since $|F|= m/r_1$, we only need to verify that the $g_i +w_j$ are uniquely determined modulo $X$, $Y$,  and $Z$.
If $g_i +w_j \equiv g_{i'} +w_{j'} \pmod X$, then $g_i -g_{i'} \in X + N = X+K$. By the conditions in \eqref{eq:Sum2-g}, this implies $i = i'$. Thus, $w_j -w_{j'} \equiv 0 \pmod X$, so $w_j - w_{j'} \in X \cap N = (L+ A) +(X \cap K)= L+ A$, where the first equality follows from Dedekind's modular law since $L +A  \leq X$ and the last equality from  $X \cap K \leq X \cap Y  \cap Z = \{0\}$. Finally, equation \eqref{eq:Sum2-w} implies that $w_j \equiv w_{j'} \bmod {L+ A}$ only if $j =j'$. This proves that $F$ is a complete system of coset representatives of $X$ as claimed. 

The arguments for the remaining two cases $\bmod Y$ and $\bmod Z$ are similar, and we omit the details. We only remark that  regarding  $Y$, we use the equations $Y+N = Y+L$ and $Y\cap N = K + A +(Y \cap L)= K + A$. For $Z$, we use $Z+N= Z+ A $ and $Z \cap N= (L+K) +(Z \cap A)= L+K$.

This concludes Case 1, demonstrating that $\T_G(X, Y, Z) \neq \emptyset$ and thus proving the inductive step. 

{\bf Case 2.} ($r_1^2 = m$): \\
In this case $G$ is of a very specific type. 
Indeed, as $m= r_1^2 \leq r_1 r_2 \leq l_1 \leq m$ and $m=r_1^2 \leq r_1 r_3 \leq l_2 \leq m$, we have $r_1 = r_2 = r_3$ and $l_1 = l_2 = l_3=m$. Thus,
\begin{align}
G &= A \oplus B \oplus C\\
X&=T_X= A \oplus B\\
Y&=T_Y= A\oplus  C\\
Z&=T_Z= B\oplus C.
\end{align}
Let  $A= \{a_i \mid i \leq r_1\}$,  $B = \{b_i \mid i \leq r_1\}$ and  $C = \{c_i \mid i \leq r_1\}$. Then the set 
\[
F= \left\{ a_i+b_i+c_i \mid i \leq r_1 \right\}
\]
 is a common transversal for $X, Y $ and $Z$ in $G$, as can be easily verified. Hence $T_G(X, Y, Z) \neq \emptyset$, as desired.

{\bf Case 3.} ($r_1= 1$):\\
This forces $G$ to be a direct sum $G = X \oplus Y$,  a case already covered in \autoref{Thm:Direct-Sum-of-Two} according to which  $T_G(X, Y, Z) \neq \emptyset $ (as $X, Y, Z$ are not all cyclic). 

This completes the inductive argument and the proof of the theorem.  
\end{proofof}

\section{ The general case for $p$-groups }\label{sec:GeneralCase}

This section considers the general case where subgroups $X$, $Y$, and $Z$ all contribute to $G$. As in \autoref{Thm:Sum-of-Two}, we proceed by induction on the index $[G:X]$, reducing to the case where $G$ is the sum of two of the subgroups. As earlier, the odd prime case is straightforward, but 
the prime $p=2$ requires special attention due to an exceptional case treated in \autoref{prop.missing-X+Y+Z-2}. Let us start with a lemma that will be needed for the proof of \autoref{prop.missing-X+Y+Z-2}.
\begin{lemma}\label{lem:X+Y+Z-hat{X}Direct}
Let $G$ be a finite abelian $2$-group. Suppose $X, Y, Z$ are proper subgroups of $G$ of order $m>2$ with trivial pairwise intersections, such that $G = X+Y+Z$. Assume further that there exists a non-trivial subgroup  $1 < K \leq   X$ with  $X = \hat{X}  \oplus K$, where  $\hat{X}= X \cap (Y+Z)$.  Then $ \T_G(X, Y, Z) \neq \emptyset$. 
\end{lemma}

\begin{proof}
In the context of the lemma, we have $r_1=r_2=r_3=1$ and $l_1=l_2=l_3=:l$ (where the notation is as in \eqref{eq:lr}). So $|G|= m^3/l$ and the index of $X$ in $G$ is $m^2/l$. 
If $l=1$ then $\hat{X} = \hat{Y} = \hat{Z} = \{0\}$ and thus $G= X \oplus Y \oplus Z$. This was the first case where the existence of a common transversal was established (see Theorem 1 in~\cite{kol97}); for completeness, we state (without proof) the actual transversal 
\[
F=\left\{ x_i+y_i+ x_j+z_j \mid i , j \leq m  \right \} 
\]
where $X= \{x_i\}, Y= \{y_i\}$ and $Z= \{ z_i \}$ for $i \leq m$.

Thus we may assume now that $l >1$.   Let $K = \{ k_i \mid i \leq m/l \}$. For $i \leq m/l$, let $\{ y_i \}$ and $\{ z_i \}$ be complete sets of coset representatives of $\hat{Y}$ and $\hat{Z}$ in $Y$ and $Z$, respectively. 
By \autoref{lem:correspondence-Transversals},  we can enumerate $\hat{X} = \{ \hat{x}_j \}$, $\hat{Y} =\{\hat{y}_j\}$, and $\hat{Z}= \{\hat{z}_j\}$ such that
\begin{equation} \label{eq:hat{X}Direct-1}
    \hat{z}_j =\hat{x}_j + \hat{y}_j 
\end{equation}
for all $j \leq l$. This works because  $T_X=T_Y = T_Z=\{0\}$, so the coset representatives in \autoref{lem:correspondence-Transversals} correspond directly to elements of $\hat{X}$, $\hat{Y}$, and $\hat{Z}$.

As $m \neq l$ we can apply  \autoref{lem:2-bijections} to $1< \hat{Y} < Y $. Thus  each element  in $Y$ can be uniquely expressed as \( y_i + \hat{y}_j \) for indices $i \leq m/l$ and $j \leq l$,  and there exists a bijection $f: Y \to Y $ such that $f(y_i +\hat{y}_j ) - \hat{y_j}$ is also a bijection on $Y$.

Now consider the following set
\[
F=\left\{ k_i +\hat{x}_j +f(y_i+\hat{y}_j) + z_t +k_t \mid  i,t \leq m/l, \, \,  j \leq l\right \}, 
\]
 of order $m^2/l = [G:X]$. We claim that $F$ is a common transversal of $X, Y, Z$.

Modulo $X$, the set is $\{ f(y_i+\hat{y}_j) + z_t \mid i,t \leq m/l, \, \,  j \leq l \}$. We need to show that all its elements are inequivalent modulo $X$. If $f(y_i+\hat{y}_j) + z_t = f(y_{i'}+\hat{y}_{j'}) + z_{t'} +x $, for some $x \in X$,  then $z_t - z_{t'} \in Z \cap (Y+X) = \hat{Z}$, implying $t = t'$. This forces $f(y_i+\hat{y}_j) = f(y_{i'}+\hat{y}_{j'}) +x $. Since $X \cap Y = \{0\}$, we have $f(y_i+\hat{y}_j) = f(y_{i'}+\hat{y}_{j'})$, which contradicts the bijectivity of $f$ unless $(i, j) = (i', j')$.

Modulo $Y$, the set $\{ k_i +\hat{x}_j + z_t +k_t \mid  i,t \leq m/l, \, \,  j \leq l \}$ also contains inequivalent elements. Indeed, if $ k_i +\hat{x}_j + z_t +k_t \equiv k_{i'} +\hat{x}_{j'} + z_{t'} +k_{t'} \pmod{Y}$, then $t=t'$ (since $z_t - z_{t'} \in \hat{Z}$, as earlier). Consequently, $k_i +\hat{x}_j \equiv k_{i'} +\hat{x}_{j'} \pmod{Y}$, implying $k_i +\hat{x}_j - (k_{i'} +\hat{x}_{j'}) \in X\cap Y =\{0\}$. Thus, $k_i +\hat{x}_j = k_{i'} +\hat{x}_{j'}$, forcing $(i,j)= (i',j')$ as the sum $X=\hat{X}\oplus K$ is direct.

Finally, modulo $Z$ and under the light of  \eqref{eq:hat{X}Direct-1},  the set becomes 
\[
\{k_i +\hat{x}_j +f(y_i+\hat{y}_j)  +k_t \}=\{k_i -\hat{y}_j +f(y_i+\hat{y}_j)  +k_t \},  
\]
for $i,t \leq m/l$ and $ j \leq l$. 
If $k_i -\hat{y}_j +f(y_i+\hat{y}_j)  +k_t= k_{i'} -\hat{y}_{j'} +f(y_{i'}+\hat{y}_{j'})  +k_{t'}+z$, then $k_i+k_{t}-k_{i'}-k_{t'} \in K \cap (Z+Y)\leq K \cap (X \cap (Z+Y) ) = K \cap \hat{X}= \{0\}$. Thus, $k_i+k_{t}-k_{i'}-k_{t'} =0$ and $f(y_i+\hat{y}_j) -\hat{y}_j \equiv  f(y_{i'}+\hat{y}_{j'})-\hat{y}_{j'} \pmod {Z}$. Since $Y \cap Z= \{0\}$, we have $f(y_i+\hat{y}_j) -\hat{y}_j = f(y_{i'}+\hat{y}_{j'})-\hat{y}_{j'}$. As $f(y_i +\hat{y}_j ) - \hat{y_j}$ is a bijection of $Y$, it follows that $(i, j)=(i',j')$. Consequently, $k_i+k_{t}-k_{i'}-k_{t'} =0$ simplifies to $k_{t}- k_{t'} =0$, implying $t=t'$.

This completes the proof of the lemma.
\end{proof}

Now we can prove:
\begin{proposition}\label{prop.missing-X+Y+Z-2}
Let $G$ be a finite abelian $2$-group. Let $X, Y, Z $ proper, not all cyclic,  subgroups of $G$ of the same order $m$, with trivial pairwise intersections and $\hat{X} \cong \hat{Y} \cong  \hat{Z} \cong  \Z_{2^n} $, for some $n \geq 1$, where $\hat{X}= X \cap (Y+Z), \, \hat{Y}= Y \cap (X+Z)$ and  $\hat{Z}= Z \cap (X+Y)$.  Then $\T_G(X, Y, Z) \neq \emptyset$.
\end{proposition}

\begin{proof} 
Without loss we may assume $X$ is non-cyclic, so $\hat{X} \neq X$. Then $\hat{X}, \hat{Y},$ and $\hat{Z}$ are all proper subgroups of $X, Y,$ and $Z$, respectively, since all families of groups (hat and non-hat)  have the same order.   Because $\mathbb{Z}_{2^n} \cong \hat{X} \leq X$, there exists a subgroup $K \leq X$ of order $2$ such that $X_1:= K \oplus \hat{X} \leq X$, and thus $|X_1| = 2^{n+1} \leq m$. 
Suppose $\hat{Y} < Y_1 \leq Y$ and $\hat{Z} < Z_1 \leq Z$ are subgroups of $Y$ and $Z$ respectively, with order $2^{n+1}$. If $X = X_1$, then by \autoref{lem:X+Y+Z-hat{X}Direct} the proposition holds. Thus,  we may assume $X_1 < X$ and $2^{n+1} < m$.

Now we look at the group $G_1:=X_1+Y_1+Z_1$. Clearly the intersection of any two of $X_1, Y_1, Z_1$  is trivial. Also 
\[
\hat{X} \leq  \hat{X}_1= X_1 \cap (Y_1 +Z_1)  \leq X \cap (Y+Z) = \hat{X}
\]
and thus $\hat{X}_1= \hat{X}$. Hence  $X_1 = \hat{X}_1 \oplus K$,   and \autoref{lem:X+Y+Z-hat{X}Direct} applies to $G_1$. 
Let us note that $|G_1|= 2^{2n+3}$ and   $[G_1: X_1]= 2^{n+2}$.
Let $\{ g_i \mid i \leq  2^{n+2}\} \leq G_1$
be a common transversal of $X_1, Y_1, Z_1$ in $G_1$, that is, 
\begin{equation*}
g_i \not\equiv g_{i'}  \bmod {X_1},  \, \, \bmod {Y_1}, \, \,  \bmod {Z_1}, 
\end{equation*}
  for all   $ 1 \leq i \neq i' \leq 2^{n+2}$.

For every $1 \leq i\leq m/2^{n+1}$, 
let $\{ x_i\}, \{y_i\}$ and $\{z_i\}$ be transversals of $X_1$, $Y_1$ and $Z_1$ in $X, Y$, and $Z$ respectively, so 
\[
X= \bigsqcup_{i \leq \frac{m}{2^{n+1}}} x_i +X_1, \, \, Y= \bigsqcup_{i \leq \frac{m}{2^{n+1}}} y_i +Y_1, \, \, Z= \bigsqcup_{i \leq \frac{m}{2^{n+1}}} z_i +Z_1. 
\]

We claim the set 
\[
F =\left\{ g_i +x_j +y_j +x_t +z_t  \mid  i \leq 2^{n+2},  \,\, j, \, t \leq m/2^{n+1} \right\}  
\]
is a common transversal for $X, Y, Z$. 
We first observe that $|F|= m^2/2^n=[G:X]$, as $|G|= m^3/2^n$.

Modulo $Z$, the set $F$ becomes $\{ g_i +x_j +y_j +x_t \mid i \leq 2^{n+2}$ and $ j, \, t \leq m/2^{n+1}\}$. To prove that its elements are inequivalent modulo $Z$, suppose $g_i +x_j +y_j +x_t = g_{i'}+x_{j'}+ y_{j'} +x_{t'} +z$ for some $z \in Z$. This implies $y_j- y_{j'} \in Y \cap (G_1 +X+Z)$. We have
\[
Y \cap (G_1 +X+Z) = Y \cap (Y_1+X+Z) = Y_1 +(Y \cap (X +Z) ) = Y_1+ \hat{Y} = Y_1, 
\]
where the second equality follows from Dedekind's modular law. Hence $y_j- y_{j'} \in Y_1$ implying  $j = j'$. 
Now $x_t - x_{t'} = g_{i'}-g_i +z \in X \cap (G_1+Z)$. But  $X \cap (G_1+Z) = X_1 + (X \cap (Y_1 +Z))  = X_1 $, where the last equality follows from   $\hat{X} \leq X \cap (Y_1 +Z) \leq X \cap (Y+Z) = \hat{X}$. 
So $x_t \equiv x_{t'} \pmod {X_1}$ forcing $t = t'$. Finally, if  $g_i -g_{i'} = z \in Z \cap G_1$ then $g_i \equiv g_{i'} \pmod {Z_1}$ (because $Z \cap G_1 = Z_1 + (Z \cap (X_1 +Y_1) ) = Z_1+\hat{Z} = Z_1$). Thus $i=i'$,  proving that the elements in $F$ modulo $Z$ are all inequivalent.

The arguments for Modulo X and Modulo Y are similar to the Z-case, so we omit them. The proposition now follows. 

\end{proof}

Now we can prove:
\begin{theorem}\label{Thm:General-case}
  Assume $G$ is an  abelian $p$-group. Let $X, Y, Z$ be proper subgroups of the same index in $G$ with $X \cap Y \cap Z= \{0\}$ and $G=X+Y+Z$.  Then $\T_G(X, Y, Z) =  \emptyset $ if and only if $p=2$, the groups $X, Y, Z$ are all cyclic and   $G = X \oplus Y = X \oplus Z = Y \oplus Z$.
\end{theorem}

\begin{proof}If $X, Y, Z$ are cyclic, the theorem follows from Theorem B in~\cite{ALS24}. Therefore, 
it suffices to show that $T_G(X, Y, Z) \neq \emptyset$ when $X, Y, Z$ are not all cyclic. Without loss of generality, assume $r_1 \leq r_2 \leq r_3$, which implies $l_1 \leq l_2 \leq l_3$,  by \eqref{eq:lr}. 

We proceed by induction on the index $[G:X]= \frac{m^2}{l_3r_1}$ to prove the existence of a common transversal in the non-cyclic case. The base case $[G:X]=p$ is established in Lemma \ref{Lem:Index-p}, initiating our induction. Assume the theorem holds for any group $H = H_1 + H_2 + H_3$, where the $H_i$ are trivially intersecting subgroups, not all cyclic, of the same order and index strictly less than $\frac{m^2}{l_3r_1}$; we aim to prove it also holds for $X, Y, Z$.

Since $\frac{l_1 r_3}{r_1} = l_3 \leq m$, there exists a subgroup $X_1$ such that $\hat{X} \leq X_1 \leq X$ and $[X_1 : \hat{X}] = \frac{r_3}{r_1}$, implying $|X_1|= l_3$. Similarly, as $l_2 r_2/r_1 = l_3 \leq m$, there exists a subgroup $Y_1 \leq Y$ with $\hat{Y} \leq Y_1$ and $[Y_1:\hat{Y}]=  \frac{r_2}{r_1}$, thus $|Y_1|= l_3$. Observe that $|X+Y_1|= |Y +X_1|=  \frac{ml_3}{ r_1}$. Furthermore, 
\[
\hat{Z} \leq Z \cap (X_1 + Y_1) \leq Z \cap (X+Y) = \hat{Z},  
\]
 where the first inequality follows from $\hat{X} + \hat{Y} \leq X_1 +Y_1$ along with equation \eqref{eq:hat-groups}. Hence,
\[
|Z +X_1 +Y_1|= \frac{m|X_1 +Y_1|}{l_3}= \frac{m l_3}{r_1}, 
\] 
 since $|X_1+Y_1|= \frac{l_3^2}{r_1}$.  The situation is described in \autoref{fig:myfigure4}.
 \begin{figure}[h]
\includegraphics[width=0.55\textwidth]{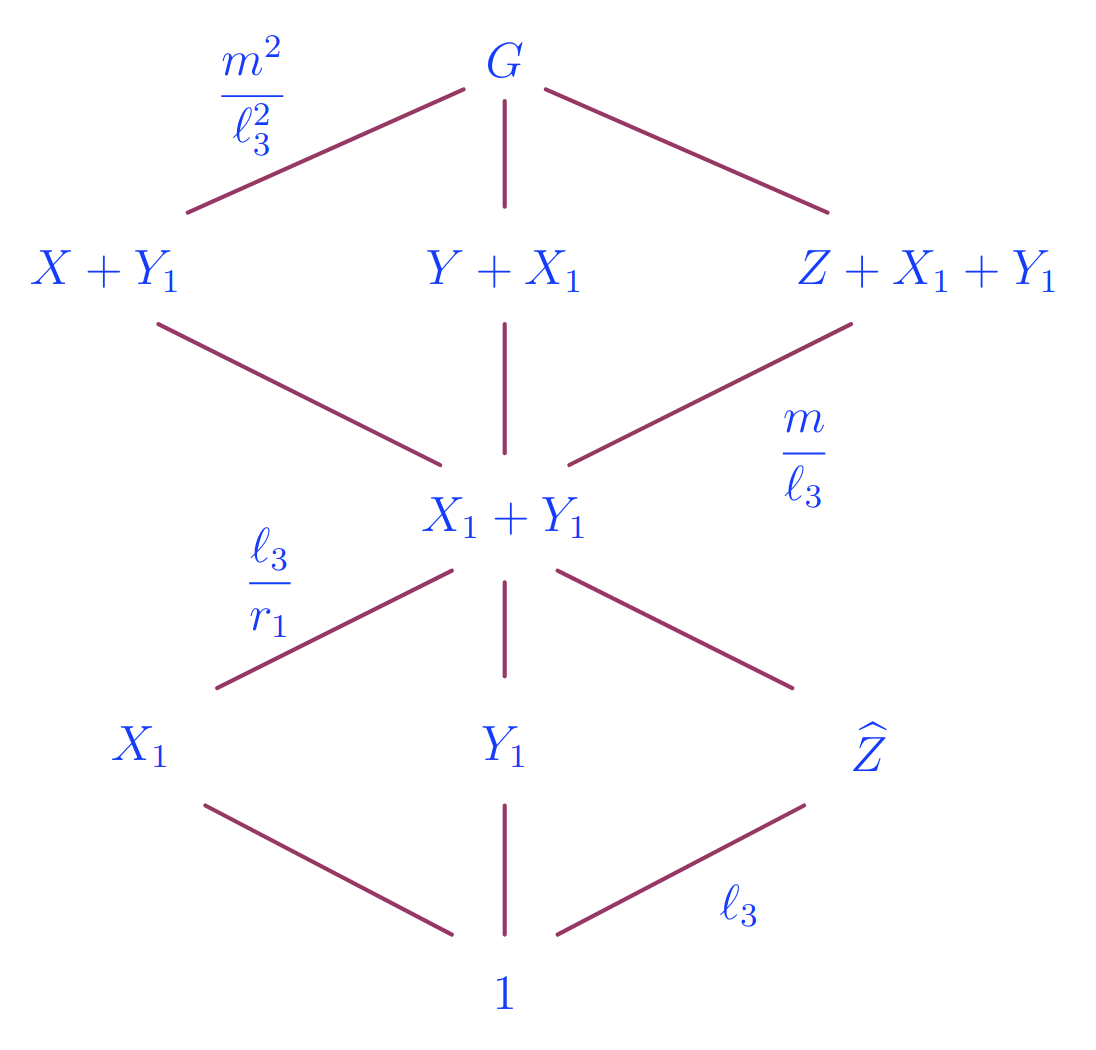}
\caption{The groups involved in this case}
\label{fig:myfigure4}
\end{figure}

 We would like to apply our inductive hypothesis twice: first to the subgroups $X+Y_1, Y +X_1, Z+X_1+Y_1$ of $G$, and second to the subgroups $X_1, Y_1, \hat{Z} $ of $X_1 +Y_1$. We distinguish  cases  for $l_3$ to ensure the inductive hypothesis is applicable. 
 
{\bf Case 1.} ($r_1 < l_3 < m$):

Let us first  consider the subgroups  $X + Y_1, Y+ X_1$ and $Z +  X_1 +  Y_1$  whose index in $G$ is  $m^2/l_3^2$ and thus they are proper (since $l_3 \neq m$) while $m^2/l_3^2 < m^2/l_3 r_1 = [G:X]$ (as $l_3 > r_1$).  
Furthermore one can show (by multiple applications of Dedekind's law), 
\[
S:= (X+Y_1) \cap (Y+X_1)  \cap (Z+X_1+Y_1) = X_1+Y_1. 
\]
Now we look at the quotient group $\widebar{G} = G/S$ and its subgroups $\widebar{X} := (X + Y_1) /S, \widebar{Y }:= (Y + X_1) /S$,
and $\widebar{Z} := (Z + X_1 +Y_1) /S$, each with index $m^2/l_3^2$ in $\widebar{G}$.
If $\widebar{X}, \, \widebar{Y}, \, \widebar{Z}$  are not all cyclic, then the inductive hypothesis applies, since their common index $m^2/l_3^2$ is strictly smaller than $[G:X]$. If, on the other hand, $\widebar{X}, \, \widebar{Y}, \, \widebar{Z}$  are all cyclic, then Theorem B of \cite{ALS24} applies to this smaller triple. Therefore, $\widebar{X}, \widebar{Y}$, and $\widebar{Z}$  fail to have
a common transversal in $\widebar{G} $ only if they are all cyclic and 
\begin{equation}\label{eq:X+Y+Z-wide}
 \widebar{G}  = \widebar{X} \oplus \widebar{Y}  = \widebar{X} \oplus \widebar{Z} = \widebar{Y} \oplus \widebar{Z}. 
\end{equation}
However, if the above equation holds, then $G = X+Y$, and \autoref{Thm:Sum-of-Two} applies to $X, Y, Z$, completing the induction.

Thus we may assume that \eqref{eq:X+Y+Z-wide} does not occur. Hence  $\T_{\bar{G}} (\bar{X}, \bar{Y}, \bar{Z}) \neq \emptyset $ and so $\T_G(X+Y_1, Y+X_1, Z+X_1+Y_1) \neq \emptyset$.  Let 
$\{ g_i  \mid i \leq m^2/ l_3^2\} $  be a common transversal of $X+Y_1, Y+X_1$ and $Z+X_1 +Y_1 $ in $G$, that is 
\begin{equation}\label{eq:Sum3-g}
g_i \not\equiv g_{i'}  \bmod {X+Y_1},  \, \, \bmod {Y+X_1}, \, \,  \bmod {Z+X_1+Y_1}, 
\end{equation}
  for all   $ 1 \leq i \neq i' \leq \frac{m^2}{l_3^2} $.
  
At this point we look at the subgroups  $X_1, Y_1, \hat{Z}$ of $X_1+Y_1$  that intersect trivially and the sum of two of them is $X_1+Y_1$. By \autoref{Thm:Sum-of-Two}, $X_1, Y_1, \hat{Z}$ share a common transversal in $X_1+Y_1$, unless  $X_1, Y_1, \hat{Z}$ are all cyclic $2$-groups and 
\begin{equation}\label{eq:X_1+Y_1}
X_1 +Y_1 = X_1 \oplus Y_1 = X_1 \oplus \hat{Z} =  Y_1 \oplus \hat{Z}.
\end{equation}
This would imply that the intersections of any two of  $X_1, Y_1$ and $\hat{Z}$ are trivial and thus  $r_1=r_2=r_3=1$. Hence   $l_1=l_2=l_3>1$ (as $l_3 > r_1=1$),  and $\hat{X}= X_1, \hat{Y}=Y_1, \hat{Z}$ are all cyclic groups.  This corresponds precisely to the exceptional situation (for non-cyclic $X, Y, Z \leq G$) handled in \autoref{prop.missing-X+Y+Z-2}, thereby establishing the induction step.
We may therefore assume that \eqref{eq:X_1+Y_1} does not hold. Consequently, there exists a common transversal $\{ w_j \mid j \leq l_3/r_1\}$ of $X_1, \, Y_1, \hat{Z}$ in $X_1 + Y_1$, namely,
\begin{equation}\label{eq:Sum3-w}
w_j  \not\equiv w_{j'}  \bmod {X_1},  \, \, \bmod {Y_1}, \, \,  \bmod {\hat{Z}}, 
\end{equation}
for all  $ 1 \leq j  \neq j' \leq l_3/r_1$.

We claim that 
\[
F = \left\{ g_i + w_j \mid i \leq m^2/l_3^2, \, j \leq l_3/r_1 \right \}
\]
is a common transversal of $X$, $Y$, and $Z$ in $G$. Since $|F|= m^2/l_3r_1$, we only need to verify that the elements $g_i +w_j$ are uniquely determined modulo $X$, $Y$,  and $Z$.

If $g_i +w_j \equiv g_{i'} +w_{j'} \pmod X$, then $g_i -g_{i'} \in X + Y_1$. By the conditions in \eqref{eq:Sum3-g}, this implies $i = i'$. Thus, $w_j -w_{j'} \equiv 0 \pmod X$, so $w_j - w_{j'} \in X \cap (X_1+Y_1)= X_1$, by  Dedekind's modular law as $(X \cap Y_1 = A \leq X_1$). Finally, equation \eqref{eq:Sum3-w} implies that $w_j \equiv w_{j'} \bmod X_1$ only if $j =j'$. Hence $g_i +w_j \not\equiv g_{i'} +w_{j'} \pmod X$ for all $(i, j) \neq (i',j')$, proving that $F$ is a complete system of coset  representatives of $X$. 

The arguments for the remaining two cases $\bmod Y$ and $\bmod Z$ are similar, and we omit the details. We only remark that  regarding  $Z$, we use the equation  $Z \cap (X_1 +Y_1)= \hat{Z}$.
This completes the proof of the induction and concludes Case 1. 

{\bf Case 2.} ($r_1=l_3$):

We have already seen,  by \eqref{eq:def.ABC} and \eqref{eq:lr},  that $r_2 r_3 \mid l_3$. So, under the assumptions $r_1 \leq r_2 \leq r_3$ and  $r_1=l_3$ we conclude that $r_1=r_2=r_3=l_1=l_2=l_3=1$. Hence 
\[
G= X \oplus Y \oplus Z.
\]
As we have already seen a common transversal,  in this case, is known to exist by  Theorem 1 in~\cite{kol97} and is given as  
\[
F=\left\{ x_i+y_i+ x_j+z_j \mid i , j \leq m  \right \} 
\]
where $X= \{x_i\}, Y= \{y_i\}$ and $Z= \{ z_i \}$ for $i \leq m$.

{\bf Case 3.} ($l_3=m$):

If $l_3 = m$, then $\hat{Z} = Z$, and $G = X+Y$, a case treated in \autoref{Thm:Sum-of-Two}. This completes the inductive argument and the  proof of the theorem.

\end{proof}

\section{Proof of \autoref{Thm:main}}\label{sec:ThmA}
To complete the proof of our main theorem, \autoref{Thm:main}, it remains to address a  missing step: passing from transversals of the subgroups $X, Y, Z$ in $G$ to transversals of their $p$-Sylow subgroups $X_p, Y_p, Z_p$ in $G_p$, and vice versa. This is precisely the content of the next lemma, and we note that Corollary 4.3 in~\cite{ALS24} is a special case of it.
\begin{lemma} \label{Lem:reduction-to-p-groups}
Assume $X, Y, Z$ are subgroups of the same order of a finite  abelian group $G= X+Y+Z$ with $X \cap Y \cap Z=\{0\}$.  Then $X, Y, Z$ share a common transversal in $G$ if and only if $X_p, Y_p$ and $Z_p$ share a common transversal in $G_p$ for all prime divisors $p$  of $|G|$.
\end{lemma}

\begin{proof}
If $X_p, Y_p, Z_p$ share a common transversal $T_p$  in $G_p$, for every prime divisor $p$ of $|G|$ then the product $T= \prod_p T_p$ is a common transversal for $X, Y, Z$ in $G$. For the converse, common transversals exist for $X_p, Y_p,$ and $Z_p$ in $G_p$ by \autoref{Thm:General-case}, except when $p=2$, $X_2, Y_2, Z_2$ are cyclic, and $G_2 = X_2\oplus Y_2 = X_2 \oplus Z_2 = Y_2 \oplus Z_2$. Thus, to prove the lemma, it suffices to show that this case for Sylow 2-groups is impossible when $X, Y, Z$ share a common transversal in $G$. We proceed by contradiction. So assume  
\begin{equation}\label{eq:finalsum}
G_2 =X_2\oplus Y_2 = X_2 \oplus Z_2 = Y_2 \oplus Z_2
\end{equation}
with  $X_2, Y_2, Z_2$ being cyclic and let  $T= \{ t_i  \mid i \leq n \} $ be a  common transversal of  $X, Y, Z$ in $G$, where $n=[G:X]$. Let $|G|= 2^k q$ where $k$ and $q$ are integers with $q$ odd. Then $n= n_2 \times n_q$ with $n_2 = [G_2 : X_2]= |X_2|$ and $n_q= [G_q :X_q]$, where $G_q$ is the Hall $q$-subgroup of $G$. Thus, each $g \in G$ can be uniquely written as $g = g_2 +g_q$, where $g_2 \in G_2$ and $g_q \in G_q$ are its $2$-part and $2'$-part, respectively.
So, for each element  $t_i \in T$, we have 
\[
t_i = a_i + b_i,
\]
where $a_i \in G_2$ and $b_i \in G_q$. Note that the $a_i$ and $b_i$ need not be distinct for different $i$. 
Consider the sequence $A= \{ a_1, a_2,  \cdots, a_n\}$ of the  $2$-elements $a_i$ taken with multiplicities. 

Fix $a \in A$. We claim that 
\[
n_q= \left|\{ i \leq n  \mid     a_i  \equiv a \pmod{X_2} \} \right|.
\]
To prove our claim first note that if $\{c_1, \cdots c_{n_q}\}$ is a transversal of $X_{q} $ in $G_q$ then the elements $\{ a+c_j  \mid j \leq n_q \} $ are clearly inequivalent modulo $X$. Thus, for each $j = 1, \cdots , n_q$, there exists $t_j=a_j +b_j   \in T$ such that $t_j \equiv a+c_j \pmod X$. Note that the $t_j$ are distinct, and the equivalence relation implies $a_j \equiv a \pmod {X_2} $ and $b_j \equiv c_j  \pmod {X_q}$. Hence  $\left|\{ i \leq n \mid  a_i  \equiv a \pmod{X_2} \} \right|
\geq n_q$. 

For the reverse inequality, assume $a_i \equiv a_j \pmod{X_2}$. Then $t_i - t_j = a_i - a_j + b_i - b_j \equiv  b_i - b_j  \pmod X$, as $X_2 \leq X$.  Since $t_i \not\equiv t_j \pmod X$, we have $b_i \not\equiv b_j \pmod{X_q}$. Because there are $n_q$ inequivalent elements of $G_q$ modulo $X_q$, there are at most $n_q$ distinct elements $t_i$ with $a \equiv a_i \pmod{X_2}$. Thus, $\left|\{ i \leq n  \mid a_i \equiv a \pmod{X_2} \} \right| \leq n_q$, and the claim follows.

But $X_2$ is no more special than $Y_2$ and $Z_2$ and thus our claim really applies to all three subgroups, i.e.,
$n_q=\left|\{ i \leq n \mid a_i \equiv a \pmod {Y_2} \} \right|=\left|\{ i \leq n \mid a_i \equiv a \pmod {Z_2} \} \right|$.

Let us next observe that any coset of $X_2$ in $G_2$ must be represented by an element of $T$ because $T$ is a transversal of $X$ in $G$.  As $G_2 = X_2 \oplus Y_2$,  $Y_2$  is a transversal of $X_2$ in $G_2$. Thus,  for every $y \in Y_2$,  there exists  $t_i \in T$ of the form   $t_i = a_i +b_i $ so that $a_i \equiv y  \pmod {X_2}$, and by our claim, there are exactly $n_q$ such $t_i \in T$.  (Indeed, consider the coset $y + X$ in $G$. Since $T$ is a transversal, there exists exactly one element in $T$ that lies in $y\pmod X$ call it $t=a + b$ for some $a\in A, b\in G_q$. Then $y \equiv t \pmod{X}$ implying that   $a\equiv y\pmod {X_2}$.)

Similarly, there are exactly $n_q$ elements $t_i \in T$ of the form $t_i = a_i +b_i $ with $a_i \equiv z \pmod {X_2}$ for every $z \in Z$, and similarly for every other 
combination of $X_2, Y_2, Z_2$.

Since $a_i \in G_2 = X_2 \oplus Y_2$, we can write $a_i = x_i + y_i$ with $x_i \in X_2$ and $y_i \in Y_2$. The observation above implies that the sequence $\{x_i \mid i \leq n\}$ contains each element of $X_2$ exactly $n_q$ times, and similarly for $\{y_i \mid i \leq n\}$ and $Y_2$. Thus,
\[
\sum_{i=1}^n a_i = \sum_{i=1}^n x_i + \sum_{i=1}^n y_i = n_q x_0 + n_q y_0,
\]
where $x_0$ and $y_0$ are the unique involutions in $X_2$ and $Y_2$, respectively.

Similarly, since $G_2 = X_2 \oplus Z_2$, we can write $a_i = x'_i + z_i$. Applying the same argument, we have
\[
\sum_{i=1}^n a_i = \sum_{i=1}^n x'_i + \sum_{i=1}^n z_i = n_q x_0 + n_q z_0,
\]
where $z_0$ is the unique involution of $Z_2$. But now we should have $n_q y_0= n_q z_0$. Since $Y_2 \cap Z_2 = \{0\}$, this forces $n_q y_0 = n_q z_0 = 0$. As $n_q$ is odd and $y_0, z_0$ are involutions, we reach a contradiction. Thus, the lemma is proven.
 
\end{proof}

The proof of our  main theorem is now immediate. 

\begin{proofofA}
 By    \autoref{Lem:G=X+Y+Z}, we may assume that $G = X+Y+Z$. Also,   according to \autoref{lem:factor} we get 
$\T_G(X, Y, Z) \neq   \emptyset $ if and only if $\T_{G/N} (X/N, Y/N, Z/N) \neq \emptyset$. 

 If $G_{p'}$ denotes the $p'$-Hall subgroup of $G$ for a prime $p$ dividing $|G|$, then $G$ decomposes as $G = G_p \times G_{p'}$. The same type of decomposition holds for $N = N_p \times N_{p'}$ and for each of the subgroups $X = X_p \times X_{p'}$, $Y = Y_p \times Y_{p'}$, and $Z = Z_p \times Z_{p'}$. Observe that $G_p/N_p$ is the Sylow $p$-subgroup of $G/N$, and, analogously, $X_p/N_p, Y_p/N_p$, and $Z_p/N_p$ are the Sylow $p$-subgroups of $X/N, Y/N$, and $Z/N$, respectively.
Now we invoke \autoref{Lem:reduction-to-p-groups} for the quotient groups $X/N, Y/N, Z/N$ and $G/N$. Hence,
\[
\T_{G/N} (X/N, Y/N, Z/N) \neq \emptyset \quad \text{if and only if} \quad \T_{G_p/N_p} (X_p/N_p, Y_p/N_p, Z_p/N_p) \neq \emptyset
\]
for each prime divisor \(p\) of \(|G|\). By \autoref{Thm:General-case}, the proof of the theorem follows.
\end{proofofA}

We conclude the session and the paper with the proof of \autoref{thm:Lattices}.

\begin{proofofB}
By  \autoref{Lem:G=X+Y+Z}, for $R=\RR^d$, it is necessary and sufficient  for the existence of a common fundamental domain of $K,L,M$ in $\RR^d$, the existence of a common fundamental domain of $K, L,M$ inside their sum $K+L+M$. Without loss  (applying an invertible linear map) we may assume $K+L+M = \ZZ^d$. Furthermore, by  \autoref{lem:factor}, for $H = K \cap L \cap M $, a common fundamental domain of $K,L,M$ in $\ZZ^d$ exists if and only if the quotient groups
$    X:= K/H, \,  Y:= L/H$ and $Z:= M/H$
have a common fundamental domain in $G:= \ZZ^d/H$.  But $G$ is a finite abelian group,  since $H$ is a full-rank lattice in $\RR^d$ (see the  comments in  the introduction).
Now  the theorem follows from  \autoref{Thm:main}.
\end{proofofB}

\section{Examples of Lattices with no common transversal}\label{sec:Lattices} 
In this section we will provide examples of full-rank lattices (of equal volume) as subgroups of $\ZZ^2$ that do not admit a common fundamental domain.

Our first such example is about full-rank lattices $K, L, M$  whose quotients  with their common intersection $H = K \cap L \cap M$  are  cyclic groups. Consider $m\geq 1$ an odd integer and define
\begin{equation*}
    K= \begin{pmatrix}
        2^n m & 0\\
        0 & 1
    \end{pmatrix}\ZZ^2, \quad L= \begin{pmatrix}
        1 & 0\\
        0 & 2^n m
    \end{pmatrix}\ZZ^2, \quad M= \{ (k,l) \in \ZZ \times m\ZZ \: \mid  \: k \equiv \frac{l}{m} \pmod {2^n}\}.   
\end{equation*}
We note that $M = \langle (1,m), \,  (0, 2^n m ) \rangle _{\ZZ}$.  
Now it is easy to see that   \[  H =2^nm\ZZ \times 2^nm\ZZ \] and 
\[
G:=\frac{ \ZZ \times \ZZ}{H}= \frac{ \ZZ \times \ZZ}{2^n m \ZZ \times 2^n m  \ZZ}\cong  \Z_{2^nm} \times \Z_{2^nm}.
\]
Furthermore,  the quotient subgroups corresponding to $K, L, M$ are all cyclic isomorphic to $\ZZ_{2^nm}$;
\begin{align}
 X&:= K/H \cong \{0\}\times \Z_{2^nm}, \\
 Y&:= L/H \cong \Z_{2^nm} \times \{0\},\\
Z&:= M/H = \{ k(1, m ) +H  \mid  k \in \ZZ \}\cong  \Z_{2^nm }.
\end{align}
The Sylow $2$-subgroups of $X, Y$ are  $X_2= \langle  (0,m) +H \rangle $ and $Y_2= \langle  (m,0) +H \rangle $, respectively. Moreover  the Sylow $2$-subgroup of $Z$ is $Z_2=\langle  m(1,m) +H \rangle $. Thus, the requirements of \autoref{thm:Lattices}   are met, confirming that $K, L, M$ do not share a common fundamental domain.
 
 The special case where $m=1$ (under the convention that $\Z_1=\{0\}$)  yields the following full-rank lattices: 
\begin{equation*}
   K=  \begin{pmatrix}
        2^n & 0 \\
        0 & 1
    \end{pmatrix}\ZZ^2 ,\quad  L= \begin{pmatrix}
        1 & 0\\
        0 & 2^n
    \end{pmatrix}\ZZ^2, \quad M=\{(k,l)\in \ZZ^2 \: | \: k\equiv l \bmod(2^n)\}.
\end{equation*}
We remark here that for $n=1$ this is exactly the construction given in~\cite[3.1]{kol97}.

 For our next example we consider full-rank lattices $K, L, M$ where the  quotients  by their intersection $H = K \cap L \cap M$ are not all cyclic subgroups. We define  
\begin{equation*}
    K:= \begin{pmatrix}
        18 & 0\\
        0 & 1
    \end{pmatrix}\ZZ^2 , \quad  L:=\begin{pmatrix}
        1 & 0\\
        0 & 18
    \end{pmatrix}\ZZ^2, \quad M:= \{ (k,l) \in 3\ZZ \times 3\ZZ \: | \: \frac{k}{3} \equiv \frac{l}{3}
    \pmod 2\}.
\end{equation*}
Observe that $M = \langle  (3,3), \, (6,0) \rangle _{\ZZ}$ and   is 
\begin{equation*}
    H= 18\ZZ \times 18 \ZZ.
\end{equation*}
So the quotient group is 
 \begin{equation*}
     G:= \ZZ^2/H \cong \Z_{18}\times \Z_{18}.
 \end{equation*}
and the subgroups  corresponding to $K, L, M$ are 
\begin{align*}
    X&:= K/H = \{0\} \times \Z_{18},\\
     Y&:= L/H = \Z_{18} \times \{0\},\\
     Z&:= \{ (3k, \, 3l) \mid  k,l \in \Z_{6}, \, \, k\equiv l\pmod 2 \} \cong \Z_3 \times \Z_6. 
\end{align*}
While $X$ and $Y$ are cyclic, $Z$ is not; however, the Sylow $2$-subgroups of all three are cyclic of order $2$. Indeed
\begin{align*}
    X_2 &= \langle  (0, 9) +H \rangle, \\ 
    Y_2 &= \langle  (9, 0) +H \rangle, \\
     Z_2 &= \langle  (9, 9) +H \rangle.
\end{align*}
Clearly $G_2 = X_2 \oplus Y_2 = X_2 \oplus Z_2 =Y_2 \oplus Z_2 $ and thus $ K, L, M$ do not share a common fundamental domain in $\ZZ^2$,  by \autoref{thm:Lattices}.

\bibliographystyle{amsalpha}
\bibliography{Bibliography}
\end{document}